\documentclass[12pt, notitlepage]{amsart}
\usepackage{latexsym, amsfonts, amsmath, amssymb, amsthm, cite}

\newtheorem{lemma}{Lemma}
\newtheorem{proposition}{Proposition}
\newtheorem{theorem}{Theorem}

\newtheorem{conjecture}{Conjecture}

\usepackage{graphicx}
\usepackage{mathrsfs}

   \DeclareFontFamily{U}{wncy}{}
    \DeclareFontShape{U}{wncy}{m}{n}{<->wncyr10}{}
    \DeclareSymbolFont{mcy}{U}{wncy}{m}{n}
    \DeclareMathSymbol{\Sh}{\mathord}{mcy}{"58}

\newcommand{\sumflat}{\sideset{}{^{\flat}}\sum}

\begin{document}

\title[Central $L$-values of quadratic twists of elliptic curves] {Moments and distribution of central $L$-values of quadratic twists of elliptic curves} 
 \author{Maksym Radziwi\l\l \   and K. Soundararajan} 
 \address{School of Mathematics \\ Institute for Advanced Study \\ 1 Einstein Drive\\ Princeton, NJ, 08540, USA} 
 \email{maksym@ias.edu}
\address{Department of Mathematics \\ Stanford University \\
450 Serra Mall, Bldg. 380\\ Stanford \\ CA 94305-2125}
\email{ksound@math.stanford.edu}
\thanks{The first author was partially supported by NSF grant DMS-1128155. The second author is partially supported by NSF grant DMS-1001068, and a Simons Investigator award from the Simons Foundation} 
  
 \maketitle

\section{Introduction} 
 
\noindent The last two decades have seen tremendous progress in understanding the moments of central 
values in families of $L$-functions.  There are now precise, and widely believed, conjectured asymptotic formulae 
for moments in several important families (see \cite{CFKRS}, \cite{DGH}, \cite{KeSn1}, \cite{KeSn2}), and these conjectures have been verified in a variety of cases 
(see for example \cite{ChLi2}, \cite{CIS}, \cite{SoY}).   Furthermore, the work of Rudnick and Soundararajan (\cite{RuSo1}, \cite{RuSo2}), together with its extension by the authors in \cite{RaSo},
 produces lower bounds of the conjectured order of magnitude for all moments larger than the first, provided a little 
 more than the first moment can be computed.   In this paper, we enunciate a complementary principle, which (roughly speaking) establishes that if one can compute a little more than a particular moment for some family of $L$-functions, then 
 upper bounds of the conjectured order of magnitude hold for all smaller moments.  Conditional on the Generalized Riemann Hypothesis, the work of Soundararajan \cite{So} together with its sharp refinement by Harper \cite{Har} establishes the conjectured upper bounds for moments in many families; our work may be viewed as an unconditional analog of such results, but for a restricted  range of moments.   We shall illustrate our method for the important and well-studied family of quadratic twists of an elliptic curve.  Here the first moment for central $L$-values is known, but the second moment can (at present) only be calculated assuming GRH (by adapting the argument of \cite{SoY}).  
 However, there is enough flexibility for us to work out an upper bound for all moments below the first.  These ideas also enable us to study the distribution of the logarithm of the central $L$-values (when these are nonzero) and establish a one sided central limit theorem; this supports a conjecture of Keating and Snaith \cite{KeSn1}, and is an analog of Selberg's theorem 
 on the normal distribution of $\log |\zeta(\frac 12+it)|$.    Finally, our work leads to a conjecture on the distribution of the order of the Tate-Shafarevich group for rank zero quadratic twists of an elliptic curve, and establishes the upper 
 bound part of this conjecture (assuming the Birch-Swinnerton-Dyer conjecture).

Let us now describe our results more precisely.  Let $E$ be an elliptic curve defined over ${\Bbb Q}$ with conductor $N$.  
 Write the associated Hasse-Weil $L$-function as 
$$ 
L(s,E) = \sum_{n=1}^{\infty} a(n)n^{-s}, 
$$ 
where the coefficients are normalized such that the Hasse bound reads $|a(n)| \le  d(n)$ for all $n$, and so the 
center of the critical strip is $\frac 12$.   Recall that $L(s,E)$ has an analytic continuation to the entire complex plane and satisfies the functional equation 
$$ 
\Lambda(s,E)  = \epsilon_E \Lambda(1-s,E), 
$$ 
where $\epsilon_E$, the root number, is $\pm 1$, and 
$$ 
\Lambda(s,E) = \Big(\frac{\sqrt{N}}{2\pi}\Big)^{s} \Gamma(s+\tfrac 12) L(s,E). 
$$ 
Throughout the paper, $d$ will denote a fundamental discriminant coprime to $2N$, and $\chi_d =(\frac{d}{\cdot})$ the 
associated primitive quadratic character.  Let $E_d$ denote the quadratic twist of the elliptic curve 
$E$ by $d$.  The twisted $L$-function associated to $d$ is 
$$ 
L(s,E_d) = \sum_{n=1}^{\infty} a(n) \chi_d(n) n^{-s}.
$$ 
If $(d,N)=1$ then $E_d$ has conductor $Nd^2$ and the completed $L$-function 
$$ 
\Lambda(s,E_d)= \Big(\frac{\sqrt{N}|d|}{2\pi}\Big)^{s} \Gamma(s+\tfrac 12) L(s,E_d)
$$ 
is entire and satisfies the functional equation 
$$ 
\Lambda(s,E_d)= \epsilon_E(d) \Lambda(s,E_d) 
$$ 
with 
$$ 
\epsilon_E(d) = \epsilon_E \chi_d(-N). 
$$ 
Note that, by Waldspurger's theorem, $L(\tfrac 12, E_d) \ge 0$. Of course $L(\tfrac 12, E_d)=0$ when 
$\epsilon_E(d)=-1$, and in this paper, we shall restrict attention to those twists with root number $1$.  Put therefore
$$ 
{\mathcal E} = \{ d: \ \ d \text{ a fundamental discriminant with } (d, 2N)=1 \text{ and } \epsilon_{E}(d)=1\}. 
$$ 

Our first result concerns the moments of $L(\frac 12, E_d)$.  Keating and Snaith \cite{KeSn1} have conjectured that 
for all real numbers $k\ge 0$, 
$$ 
\sum_{\substack{{|d|\le X} \\ {d\in {\mathcal E}}}} L(\tfrac 12,E_d)^{k} \sim C_0(k,E) X(\log X)^{\frac{k(k-1)}{2}},
$$ 
for a specified non-zero constant $C_0(k,E)$.   As indicated earlier, this 
conjecture is known for $k=1$, and on GRH for $k=2$.   We establish a sharp upper bound for all $k$ 
between $0$ and $1$. 

\begin{theorem}  \label{thm1} 
Let $0\le k\le 1$ be a real number.   For all large $X$ we have 
$$ 
\sum_{\substack{{|d|\le X} \\ {d\in {\mathcal E}}}} L(\tfrac 12, E_d)^k \le C(k,E) X (\log X)^{\frac{k(k-1)}{2}}, 
$$ 
for a positive constant $C(k,E)$.  
\end{theorem} 

By choosing $k$ small but positive in Theorem \ref{thm1}, we see that $L(\tfrac 12, E_d) = O((\log |d|)^{-\frac 12+\epsilon})$  
for all but $o(X)$ fundamental discriminants $|d|\le X$ with $d\in {\mathcal E}$.   More is expected to be true, and Keating 
and Snaith \cite{KeSn2} have conjectured that, for $d\in {\mathcal E}$,  the quantity $\log L(\tfrac 12,E_d)$ has a  normal 
distribution with mean $-\frac 12 \log \log |d|$ and variance $\log \log |d|$; see \cite{CKRS} for numerical data towards 
this conjecture.   Here we interpret $\log L(\tfrac 12, E_d)$ 
to be negative infinity when the $L$-value vanishes.  This conjecture is an analog of Selberg's theorem on the 
normal distribution of $\log |\zeta(\frac 12+it)|$.  However the Keating-Snaith conjecture appears quite 
difficult to prove; for example, it implies the well-known conjecture of Goldfeld \cite{Gold1} that $L(\tfrac 12, E_d) \neq 0$ for 
almost all $d\in {\mathcal E}$.  We are able to prove part of the Keating-Snaith conjecture, and establish 
the conjectured upper bound for the proportion of $d\in {\mathcal E}$ with $\log L(\tfrac 12, E_d) + \frac 12\log \log |d| 
\ge V \sqrt{\log \log |d|}$ for any fixed real number $V$. 

\begin{theorem} \label{thm2} Let $V$ be a fixed real number.  For large $X$ we have 
$$ 
\Big| \Big\{ d\in {\mathcal E}, \ 20  < |d|\le X: \ \  \frac{\log L(\frac 12,E_d)+\frac 12 \log \log |d|}{\sqrt{\log \log |d|}} \ge V\Big\} \Big| 
$$ 
is at most 
$$
 | \{ d\in {\mathcal E}, |d|\le X\}| \Big( \frac{1}{\sqrt{2\pi}} \int_{V}^{\infty} e^{-\frac{x^2}2} dx +o(1)\Big). 
$$
\end{theorem} 

The connection between moments and analogs of Selberg's theorem for 
central values in families of $L$-functions is discussed in \cite{So}, where the possibility 
of establishing upper bounds as in Theorem \ref{thm2} conditional on the Generalized Riemann Hypothesis 
is mentioned.  Furthermore, if in addition to GRH one assumes the one level density conjectures of Katz and Sarnak \cite{KaSa} 
then the full Keating-Snaith conjecture on the normality of $\log L(\tfrac 12,E_d)$ would follow.  In certain families of $L$-functions, Hough \cite{Ho} has 
established unconditionally analogs of our Theorem 2.   Hough's method relies on zero density results which show that 
most zeros of the $L$-functions under consideration lie near the critical line.  While such zero density results are known in  a 
number of cases, the family of quadratic twists of an elliptic curve is an example where such results remain elusive.   The method 
used here is different (and perhaps simpler) and uses only knowledge about the first moment (plus epsilon) in the family.  

In view of the Birch and Swinnerton-Dyer conjectures, our work contributes to the understanding of the distribution of 
the order of the Tate-Shafarevich group for those quadratic twists with analytic rank zero.  Define, for $d\in {\mathcal E}$,  
$$ 
S(E_d) = L(\tfrac 12, E_d) \frac{|E_d({\Bbb Q})_{\text{tors}}|^2}{\Omega(E_d) \text{Tam}(E_d)},  
$$ 
where $|E_d({\Bbb Q})_{\text{tors}}|$ denotes the size of the rational torsion group of $E_d$, $\Omega(E_d)$ denotes the real period of a minimal model for $E_d$, and $\text{Tam}(E_d) = \prod_{p} T_p(d)$ is the product of the Tamagawa numbers.   If $L(\tfrac 12, E_d) \neq 0$ then the Birch and Swinnerton-Dyer conjecture predicts that $S(E_d)$ equals the order of 
the Tate-Shafarevich group $\Sh(E_d)$.    Now $\Omega(E_d)$ is about size $1/\sqrt{|d|}$, and the Tamagawa factors 
$T_p(d)$ are generically $1$ and for $p|d$ equal one more than the number of roots of $f(x)  \pmod p$ if $E$ is 
represented in Weierstrass form as $y^2=f(x)$.   Thus, the behavior of these quantities for large $|d|$ is relatively straightforward, and combining this with the Keating-Snaith conjecture for $L(\tfrac 12,E_d)$, we are led to formulate the following conjecture.

\begin{conjecture} \label{Conj1}  Let $E$ be given by the model $y^2=f(x)$ for a monic cubic polynomial $f$ with integer 
coefficients.  Let $K$ denote the splitting field of $f$ over ${\Bbb Q}$.   Define the constants $\mu(E)$ and $\sigma(E)$ as 
follows:  If $K={\Bbb Q}$ so that $E$ has full $2$-torsion, set 
$$ 
\mu(E) = -\tfrac 12 - 2 \log 2, \qquad \sigma(E)^2 = 1+ 4 (\log 2)^2. 
$$ 
If $[K:{\Bbb Q}]=2$ so that $E$ has partial $2$-torsion, set 
$$ 
\mu(E) = -\tfrac 12 -\tfrac 32 \log 2, \qquad \sigma(E)^2 = 1+ \tfrac 52 (\log 2)^2. 
$$ 
If $[K:{\Bbb Q}]=3$, then set 
$$ 
\mu(E) = -\tfrac 12 -\tfrac 23 \log 2, \qquad \sigma(E)^2 = 1+ \tfrac 43 (\log 2)^2. 
$$ 
Lastly if $[K:{\Bbb Q}]=6$, then set 
$$ 
\mu(E) =-\tfrac 12 - \tfrac 56 \log 2, \qquad \sigma(E)^2 = 1 + \tfrac 76 (\log 2)^2. 
$$ 
As $d$ ranges over ${\mathcal E}$, the distribution of $\log (|\Sh(E_d)|/\sqrt{|d|})$ is approximately 
Gaussian with mean $\mu(E) \log \log |d|$ and variance $\sigma(E)^2 \log \log |d|$.   More precisely, for any fixed $V \in {\Bbb R}$ 
and as $X\to \infty$, 
$$ 
\Big| \Big\{ d\in {\mathcal E}, 20< |d|\le X: \ \ \frac{ \log (|\Sh(E_d)|/\sqrt{|d|}) - \mu(E) \log \log |d|}{\sqrt{\sigma(E)^2 \log \log |d|}} 
 \ge V \Big\} \Big| 
$$ 
is 
$$ 
\sim | \{ d\in {\mathcal E}, |d|\le X\}  \Big(\frac{1}{\sqrt{2\pi}} \int_{V}^{\infty} e^{-\frac{x^2}{2}} dx \Big). 
$$   
\end{conjecture} 

Previously, Delaunay \cite{Del1} has studied the moments of orders of Tate-Shafarevich groups, and 
formulated analogs of the Keating-Snaith conjectures for the average values of $S(E_d)^k$.  Our conjecture 
is naturally closely related to his work; see also the related papers \cite{Del2} and \cite{DelWat}.   
In support of our conjecture, we are able to establish an upper bound for the distribution of values of $\log (S(E_d)/\sqrt{|d|})$ 
(as before, interpreting this quantity as $-\infty$ if $L(\tfrac 12,E_2)=0$).   

\begin{theorem} \label{thm3}  With notations as above, for fixed $V\in {\Bbb R}$ and as $X\to \infty$, 
$$ 
\Big| \Big\{ d\in {\mathcal E}, 20< |d|\le X: \ \ \frac{ \log (|S(E_d)|/\sqrt{|d|}) - \mu(E) \log \log |d|}{\sqrt{\sigma(E)^2 \log \log |d|}} 
 \ge V \Big\} \Big| 
$$ 
is bounded above by 
\begin{equation} 
\label{1.1} 
| \{ d\in {\mathcal E}, |d|\le X\}  \Big(\frac{1}{\sqrt{2\pi}} \int_{V}^{\infty} e^{-\frac{x^2}{2}} dx +o(1)\Big). 
\end{equation} 
If the Birch and Swinnerton-Dyer conjecture for elliptic curves with analytic rank zero holds then the 
quantity in \eqref{1.1} is an upper bound for  
$$ 
\Big| \Big\{ d\in {\mathcal E}, 20< |d|\le X: \ \ L(\tfrac 12, E_d)\neq 0, \ \ 
 \frac{ \log (|\Sh(E_d)|/\sqrt{|d|}) - \mu(E) \log \log |d|}{\sqrt{\sigma(E)^2 \log \log |d|}} 
 \ge V \Big\} \Big|.
 $$
\end{theorem} 

A lot of progress has been made on establishing the Birch and Swinnerton-Dyer conjecture 
in the analytic rank zero case, and thus the assumption in the 
final statement of our Theorem above seems plausibly within the reach of current technology (see \cite{Coates} for results in a 
particular family of quadratic twists, and \cite{Miller} for recent numerical verifications).

Our method is flexible enough to allow the introduction of a sieve over the fundamental discriminants $d$; thus we are 
able to obtain sharp upper bounds for moments of $L(\tfrac 12, E_d)$ where the discriminants are restricted to prime 
values of $|d|$.  Below, define 
$$ 
{\mathcal E}^{\prime} = \{ d \in {\mathcal E}: |d| \text{ is prime} \}. 
$$ 
%
%
For the twists by these ``prime" discriminants, the effect of the Tamagawa numbers in the Birch and Swinnerton-Dyer conjectures 
is negligible, and we have the following analog of the Keating-Snaith conjecture and Conjecture 1 above.   

\begin{conjecture} \label{conj2}  As $d$ ranges over ${\mathcal E}^{\prime}$, the quantities $\log (|\Sh(E_d)|/\sqrt{|d|})$ and $\log L(\frac 12,E_d)$  are distributed like a Gaussian random variable 
with mean $-\frac 12 \log \log |d|$ and variance $\log \log |d|$.  More precisely, for any 
fixed $V\in {\Bbb R}$ and as $X\to \infty$, 
$$ 
\Big| \Big \{ d\in {\mathcal E}^{\prime}: 20 < |d| \le X: \  \ \frac{\log (|\Sh(E_d)|/\sqrt{|d|}) + \frac 12 \log \log |d|}{\sqrt{\log \log |d|}} 
\ge V \Big\} \Big| 
$$ 
and 
$$ 
\Big| \Big\{ d \in {\mathcal E}^{\prime}: 20 < |d| \le X: \ \ \frac{ \log L(\frac 12,E_d) + \frac 12\log \log |d|}{\sqrt{\log \log |d|} } 
\ge V \Big \} \Big| 
$$ 
are both  
$$ 
\sim | \{ d\in {\mathcal E}^{\prime}, |d| \le X \} \Big( \frac{1}{\sqrt{2\pi} } \int_{V}^{\infty} e^{-\frac{x^2}{2} } dx\Big). 
$$ 
\end{conjecture}

Analogously to Theorems \ref{thm1}, \ref{thm2} and \ref{thm3}, our methods would enable us to obtain sharp upper bounds 
for the $k$-th moment (with $0\le k\le 1$) in this family, and also the upper bound part of Conjecture 2 (unconditionally for the distribution of $\log L(\frac 12, E_d)$, and restricted to twists with analytic rank zero and conditional on Birch and Swinnerton-Dyer for $\log (|\Sh(E_d)|/\sqrt{|d|})$).   We shall address these problems in a sequel paper.  

With minor modifications, our work applies to the family of quadratic twists of any modular form.  By Waldspurger's theorem, thus 
we may obtain an understanding of the Fourier coefficients of half-integer weight modular forms.  Further, we could also consider the family of quadratic twists where the root number is $-1$, and study here the moments of the derivative $L^{\prime}(\tfrac 12, E_d)$.   As mentioned earlier, the method developed here is general and whenever some moment (plus epsilon) is known 
in a family, our method produces sharp upper bounds for all smaller moments.  For the Riemann zeta-function, where 
the fourth moment (plus epsilon) is known (see \cite{HY}), we are thus able to establish sharp upper bounds for all 
moments below the fourth;  previously such bounds were established by Heath-Brown \cite{HB3} conditional on the Riemann Hypothesis.   Another application of this circle of ideas is to the problem of the fluctuations of a quantum observable for 
the modular surface.   More precisely, let $\psi$ denote a fixed even Hecke-Maass form for $X=PSL_2({\Bbb Z})\backslash {\Bbb H}$, and let $\phi_j$ denote an even Hecke-Maass form with eigenvalue $\lambda_j=\tfrac 14 +t_j^2$.  The problem is to understand the behavior of $\mu_{j}(\psi)=\int_{X} \psi(z) |\phi_j(z)|^2 \frac{dx \ dy}{y^2}$ as $\lambda_j$ gets large.  The mean of this quantity is approximately zero, and its variance is calculated in Zhao \cite{Zhao} (see also \cite{LS} for a holomorphic analog).   It has been suggested in the physics literature that $\mu_j(\psi)$ has Gaussian fluctuations (see \cite{EFKAMM}).  However, by Watson's formula, $|\mu_j(\psi)|^2$ is related to the central value $L(\tfrac 12,\psi\times \phi_j \times \phi_j)$ 
and the Keating-Snaith conjectures strongly suggest that $\mu_j(\psi)$ is not Gaussian (but instead $\log |\mu_j(\psi)|$ is).  This is another instance where only the first moment (plus epsilon) can be calculated, and our work would give sharp upper bounds for all moments up to the first, and establish a one sided central limit theorem for $\log |\mu_j(\psi)|$.  In particular, it would follow that $\lambda_j^{\frac 14}|\mu_j(\psi)|= o(1)$ for almost all eigenfunctions with $\lambda_j \le \lambda$.  

It would be interesting to obtain lower bounds towards the Keating-Snaith conjectures, complementing the upper bounds established here.  In work in progress, we have extended the ideas developed here to obtain a partial result in that direction provided one can control two moments in the family under consideration.  Unfortunately this does not apply to the family of quadratic twists of an elliptic curve, but would apply for example to the family of quadratic Dirichlet $L$-functions, or to the family of newforms of weight $2$ and large level $N$.  Finally, we comment that the method developed here is related to the iterative method of Harper \cite{Har} (discovered independently) which yields sharp conditional estimates for moments.   



\section{Two technical propositions}

\noindent We begin by introducing some notation that will be in place throughout the paper.  
Let $N_0$ denote the lcm of $8$ and $N$.  Let $\kappa = \pm 1 $, and let $a \pmod {N_0}$ denote a 
residue class with $a\equiv 1$ or $5\pmod{8}$.  We assume that $\kappa$ and $a$ are such that 
for any fundamental discriminant $d$ of sign $\kappa$ with $d\equiv a\pmod{N_0}$, the root number $\epsilon_E(d)=\epsilon_E \chi_{d}(-N)$ equals $1$.  
Put 
$$ 
{\mathcal E}(\kappa, a) = \{ d \in {\mathcal E}: \kappa d >0, \ \ d\equiv a \pmod{N_0}\}, 
$$ 
so that ${\mathcal E}$ is the union of all such sets ${\mathcal E}(\kappa, a)$.  Note that if $d\equiv a\pmod{N_0}$ then $d$ is automatically $1\pmod 4$ 
so that the condition of being a fundamental discriminant is simply that $d$ is squarefree.  Further, note that for $d\in {\mathcal E}(\kappa, a)$ the 
values $\chi_d(-1)$, $\chi_d(2)$, and $\chi_d(p)$ for all $p|N$ are fixed.  Therefore, it is well defined (and convenient) to set, 
for Re$(s)>0$, 
\begin{equation}\label{L_a} 
L_a(s) =  \sum_{\substack{{n=1} \\ {p|n \implies p|N_0} } }^{\infty} \frac{a(n)}{n^s} \chi_d(n).
\end{equation} 
Lastly, let $\Phi$ denote a smooth, non-negative function compactly supported on $[1/2,5/2]$ with $\Phi(x) =1$ 
for $x\in [1,2]$, and define, for any complex number $s$, 
\begin{equation} \label{Phicheck}
{\check \Phi}(s) = \int_{0}^{\infty} \Phi(x)x^{s}dx.
\end{equation}  
Throughout the paper, implied constants may depend upon $E$ (and thus $N_0$) and $\Phi$.

Our theorems rely upon two technical propositions which allow us to compute averages of 
short Dirichlet polynomials, as well as averages of $L(\tfrac 12,E_d)$ multiplied by short Dirichlet polynomials. 

\begin{proposition} \label{PropDirpoly}   Let $n$ and $v$ be positive integers both coprime to $N_0$, 
with $v$ square-free, and $(n,v)=1$.    Suppose that $v\sqrt{n} \le X^{\frac 12-\epsilon}$.  If $n$ is a square then 
$$ 
\sum_{\substack{{d \in {\mathcal E}(\kappa,a)} \\ {v|d}}} \chi_d(n) \Phi\Big(\frac{\kappa d}{X}\Big) 
= {\check \Phi}(0) \frac{X}{vN_0} \prod_{p|nv} \Big(1+\frac 1p\Big)^{-1} \prod_{p\nmid N_0} \Big( 1-\frac 1{p^2}\Big) + O(X^{\frac 12+\epsilon} \sqrt{n} ). 
$$ 
If $n$ is not a perfect square, then 
$$ 
 \sum_{\substack{{d \in {\mathcal E}(\kappa,a)} \\ {v|d}}} \chi_d(n) \Phi\Big(\frac{\kappa d}{X}\Big) = O( X^{\frac 12+\epsilon}\sqrt{n}). 
 $$ 
\end{proposition}

\begin{proposition}  \label{Prop1} Let $u$ and $v$ be positive integers with $(u,v)=1$, $(uv,N_0)=1$ and $v$ square-free.  
Define 
 \begin{equation} 
\label{S} 
{\mathcal S}(X;u,v) = \sum_{\substack{{d \in {\mathcal E}(\kappa,a)} \\ {v|d}}}  L(\tfrac 12, E_d) \chi_d(u) \Phi\Big(\frac{\kappa d}{X}\Big). 
\end{equation} 
Write $u=u_1u_2^2$ with $u_1$ square free.  Then  
$$ 
{\mathcal S}(X;u,v) =  \frac{2Xa(u_1)}{vu_1^{\frac 12} N_0} {\check \Phi}(0) L_a(\tfrac 12) L(1,\text{sym}^2 E) 
{\mathcal G}(1;u,v) + O(X^{\frac 78+\epsilon}u^{\frac 38} v^{\frac 14}). 
$$ 
Here ${\mathcal G}(1;u,v)$ may be expressed as $Cg(u)h(v)$ where $C=C(E)$ is a non-zero 
constant, and $g$ and $h$ are multiplicative functions with $g(p^k) =1+O(1/p)$, and $h(p)=1+O(1/p)$.  
\end{proposition}

The constant $C$ and the functions $g$ and $h$ are described explicitly in the proof given in section 10.   
For our work here, we only need Proposition \ref{Prop1} in the case $v=1$, but the 
general version above gives us the flexibility to introduce a sieve for the values of $d$, and thus 
enables us to obtain results over `prime' discriminants; we will discuss this problem elsewhere.   
We postpone the proofs of these propositions to sections 7 and 10, and proceed now to outline the 
proofs of the main theorems.

\section{Proof of Theorem \ref{thm1}} 

 \noindent To prove Theorem \ref{thm1} we first obtain good bounds for $L(\tfrac 12, E_d)^k$ in terms of 
 a suitable short Dirichlet polynomial and $L(\frac 12,E_d)$ times another short Dirichlet polynomial.  
 We next formulate a general such inequality. 
    
\subsection{The key inequality}  

Let $\ell$ be a non-negative integer, and $x$ a real number.  Define  
\begin{equation} 
\label{E_ell} 
E_{\ell}(x) = \sum_{j=0}^{\ell} \frac{x^{j}}{j!}. 
\end{equation} 

\begin{lemma} \label{E1}  Let $\ell$ be a non-negative even integer.  The function $E_{\ell}(x)$ is positive valued  
and convex.   Further, for any $x\le 0$ we have $E_{\ell}(x) \ge e^{x}$.  Finally, if $\ell$ is a positive even integer and 
$x\le \ell/e^2$, we have 
$$ 
e^{x} \le \Big(1 + \frac{e^{-\ell}}{16}\Big) E_{\ell}(x). 
$$
\end{lemma} 
\begin{proof}  We prove the first assertion by induction on $\ell$, the case $\ell=0$ being 
clear.  Since $E_{\ell}^{\prime\prime}(x)=E_{\ell-2}(x)$, it suffices to prove that $E_{\ell}(x)$ 
takes on positive values, and convexity follows at once.   Consider a point $x$ where 
$E_{\ell}$ takes a local minimum.   Then $E_{\ell}^{\prime}(x) =E_{\ell-1}(x)=0$, so that 
$E_{\ell}(x) =E_{\ell-1}(x)+\frac{x^{\ell}}{\ell!} = \frac{x^{\ell}}{\ell!} >0$, as desired.   
The second assertion that $E_{\ell}(x) \ge e^{x}$ for $x\le 0$ follows similarly 
by considering a local minimum for $E_{\ell}(x)-e^{x}$ on $(-\infty,0)$.  

Now we prove the final assertion, and we may assume that $0\le x\le \ell/e^2$.   Using $\ell! \ge e(\ell/e)^{\ell}$, we 
see that 
$$
e^{x}-E_{\ell}(x) \le \sum_{j=\ell+1}^{\infty} \frac{x^j}{j!} 
\le \frac{x^{\ell}}{\ell!} \sum_{j=\ell+1}^{\infty} \Big(\frac{x}{\ell}\Big)^{j-\ell} \le 
\frac{1}{6} \frac{x^{\ell}}{\ell!} \le \frac{1}{16} e^{-\ell}, 
$$ 
and since $E_{\ell}(x)\ge 1$ for $x\ge 0$, the lemma follows.
\end{proof} 

\begin{lemma} \label{E2}  Let $y\ge 0$ be a real number.  Suppose that 
$x_1$, $\ldots$, $x_R$ are real numbers, and $\ell_1$, $\ldots$, $\ell_R$ are 
positive even integers.  Then, for any $0\le k\le 1$ we have 
\begin{align*}
y^{k} &\le Cky \prod_{j=1}^{R} E_{\ell_j}((k-1)x_j) + C(1-k) \prod_{j=1}^{R} E_{\ell_j}(kx_j) 
\\
&+ \sum_{r=0}^{R-1} \Big( Cky \prod_{j=1}^{r} E_{\ell_j}((k-1)x_j) + C(1-k) 
\prod_{j=1}^{r} E_{\ell_j}(kx_j) \Big) \Big(\frac{e^2 x_{r+1}}{\ell_{r+1}}\Big)^{\ell_{r+1}}, 
\end{align*} 
where $C= \exp((e^{-\ell_1}+\ldots+e^{-\ell_R})/16)$. 
\end{lemma} 
\begin{proof}  Suppose first that $|x_{j}| \le \ell_j/e^2$ for all $1\le j\le R$.   Recall Young's inequality: 
if $a$ and $b$ are non-negative and $p\ge 1$ with $1/p+1/q=1$ then $ab\le a^p/p + b^q/q$.   Using this with $p=1/k$, $q=1/(1-k)$, $a=y^k \exp(k(k-1)(x_1+\ldots+x_R))$ and $b=\exp(k(1-k)(x_1+\ldots+x_R))$, we obtain 
$$ 
y^k \le k y \exp((k-1)(x_1+\ldots+x_R)) + (1-k) \exp(k (x_1+\ldots+x_R)).
$$
Since $0\le k\le 1$ and $|x_j| \le \ell_j/e^2$, it follows that $e^{(k-1)x_j} \le (1+e^{-\ell_j}/16) E_{\ell_j}((k-1)x_j)$ 
and that $e^{kx_j} \le (1+e^{-\ell_j}/16) E_{\ell_j}(kx_j)$.   Using these inequalities we obtain 
$$ 
y^k \le Cky \prod_{j=1}^{R} E_{\ell_j}((k-1)x_j) + C(1-k) \prod_{j=1}^{R} E_{\ell_j}(kx_j). 
$$ 
This is one of the terms in the right hand side of our claimed inequality, and since all the terms 
are non-negative, the desired estimate follows in this case.  

Now suppose that there exists $0\le r\le R-1$ such that 
$|x_{j}| \le \ell_j/e^2$ for all $j\le r$, but $|x_{r+1}| > \ell_{r+1}/e^2$.  As before, using Young's inequality 
we obtain 
\begin{align*}
y^{k} &\le ky \exp((k-1)(x_1+\ldots+x_r)) + (1-k) \exp(k(x_1+\ldots+x_r)) \\
&\le Cky \prod_{j=1}^{r} E_{\ell_j}((k-1)x_j) + C(1-k) \prod_{j=1}^{r} E_{\ell_j}(kx_j). 
\end{align*}
Since $|x_{r+1}|>\ell_{r+1}/e^2$ by assumption, multiplying the right hand side by 
$(e^2 x_{r+1}/\ell_{r+1})^{\ell_{r+1}}$ only increases that quantity, and so our 
desired inequality follows in this case also. 
\end{proof}

\subsection{Estimating $L(\tfrac 12,E_d)^k$}  

We now specialize Lemma 2 to the situation at hand.  Let $d$ be an element of ${\mathcal E}(\kappa,a)$.  
Let $R$ be a natural number and 
$\ell_1$, $\ldots$, $\ell_R$ be even natural numbers.  Let $P_1$, $\ldots$, $P_R$ be disjoint subsets 
of primes $p$ not dividing $N_0$.  Define 
\begin{equation} 
\label{defP} 
{\mathcal P}_j(d) = \sum_{p\in P_j} \frac{a(p)}{\sqrt{p}} \chi_d(p). 
\end{equation} 
Given a real number $0\le k \le 1$, put 
\begin{equation} 
\label{defA} 
{\mathcal A}_j(d) = E_{\ell_j} ((k-1) {\mathcal P}_j(d)), 
\end{equation} 
and 
\begin{equation} 
\label{defB} 
{\mathcal B}_j(d) = E_{\ell_j}(k {\mathcal P}_j(d)). 
\end{equation}  

\begin{proposition} \label{Prop2}  With notations as above, we have 
\begin{align*} 
\Big(\frac{L(\frac 12,E_d)(\log |d|)^{\frac 12}}{L_a(\frac 12)} \Big)^k &\le  Ck \frac{L(\frac12, E_d)(\log |d|)^{\frac 12}}{L_a(\frac 12)}
\Big( \prod_{j=1}^{R} {\mathcal A}_j(d) + \sum_{r=0}^{R-1} \prod_{j=1}^{r} {\mathcal A}_j(d) \Big( \frac{e^2 {\mathcal P}_{r+1} (d)}{\ell_{r+1}} \Big)^{\ell_{r+1}} \Big) 
\\ 
&+ C(1-k) \Big( \prod_{j=1}^{R} {\mathcal B}_j(d) + \sum_{r=0}^{R-1}\prod_{j=1}^{r} {\mathcal B}_j(d) \Big(\frac{e^2 {\mathcal P}_{r+1}(d)}{\ell_{r+1}}\Big)^{\ell_{r+1}}\Big), 
\end{align*}
where $C= \exp((e^{-\ell_1}+\ldots+e^{-\ell_R})/16)$, as in Lemma \ref{E2}. 
\end{proposition}  
\begin{proof}  The Proposition follows upon applying Lemma \ref{E2} with $y=L(\frac 12,E_d)(\log |d|)^{\frac 12}/L_a(\frac 12)$, and $x_j ={\mathcal P}_j(d)$.  
\end{proof} 

\subsection{Estimation of terms arising from the key inequality}  

Suppose now that $X$ is large and $X/2 \le |d|\le 5X/2$.  Define a sequence of 
even natural numbers $\ell_j$ by setting $\ell_1= 2\lceil 100 \log \log X\rceil$ and 
for $j\ge 1$ put $\ell_{j+1} = 2 \lceil 100 \log \ell_j \rceil$.   Let $R$ be the largest natural number with $\ell_R >10^4$.  Note that 
the sequence $\ell_j$ is monotone decreasing for $1\le j\le R$, and indeed we have $\ell_{j} > \ell_{j+1}^2$ in this range. 
Now define ${ P}_1$ to be the set of primes below $X^{1/\ell_1^2}$ that do not divide $N_0$.  For $2\le j\le R$ define 
${ P_j}$ to be the primes lying in the interval $(X^{1/\ell_{j-1}^2}, X^{1/\ell_j^2}]$.   Next define ${\mathcal P}_j(d)$, ${\mathcal A}_j(d)$ and 
${\mathcal B}_j(d)$ as in \eqref{defP}, \eqref{defA} and \eqref{defB} above.    We shall invoke Proposition \ref{Prop2} with this choice of parameters, and use Propositions \ref{PropDirpoly} and \ref{Prop1} to estimate the terms that arise.  

\begin{proposition} \label{Prop3} With notations as above, 
$$ 
\sum_{d\in {\mathcal E}(\kappa,a)} \Big( \prod_{j=1}^{R} {\mathcal B}_j(d) + 
\sum_{r=0}^{R-1} \prod_{j=1}^r {\mathcal B}_j(d) \Big(\frac{e^2 {\mathcal P}_{r+1}(d)}{\ell_{r+1}}\Big)^{\ell_{r+1}} \Big) \Phi\Big(\frac{\kappa d}{X}\Big) 
\ll X (\log X)^{\frac{k^2}{2}}. 
$$ 
\end{proposition} 

\begin{proposition} \label{Prop4} With notations as above,  
$$ 
\sum_{d\in {\mathcal E}(\kappa,a)} L(\frac 12,E_d) \Big(  \prod_{j=1}^{R} {\mathcal A}_j(d) + 
\sum_{r=0}^{R-1} \prod_{j=1}^r {\mathcal A}_j(d) \Big(\frac{e^2 {\mathcal P}_{r+1}(d)}{\ell_{r+1}}\Big)^{\ell_{r+1}} \Big) \Phi\Big(\frac{\kappa d}{X}\Big) 
\ll X (\log X)^{\frac{k^2-1}{2}}. 
$$ 
\end{proposition} 
 
 The implied constants in Propositions \ref{Prop3} and \ref{Prop4} depend only on $k$, $\Phi$, and $E$.   We defer the 
 proofs of these propositions to sections 8 and 9, and now complete the proof of Theorem \ref{thm1}.

\subsection{Completing the proof of Theorem \ref{thm1}}  

 Applying Propositions \ref{Prop2}, \ref{Prop3}, and \ref{Prop4} we obtain 
that 
$$ 
\sum_{d\in {\mathcal E}(\kappa,a)} L(\tfrac 12,E_d)^{k} \Phi\Big(\frac{\kappa d}{X}\Big) \ll X (\log X)^{\frac{k(k-1)}2}.
$$
Now summing over the different possibilities for $a$ and $\kappa$, and breaking the range $|d|\le X$ into 
dyadic blocks, we obtain Theorem 1.

 \section{Proof of Theorem \ref{thm2}} 
 
 \noindent We begin with a well-known result on the average size of $a(p)^2$, which will be 
 useful throughout the paper.  The proof of the lemma follows from the Rankin-Selberg theory for $L(s,E)$; see 
 Chapter 5 of Iwaniec and Kowalski \cite{IK}.

\begin{lemma} \label{RS}  There exists a positive constant $c$ such that 
$$ 
\sum_{p\le x} a(p)^2 \log p = x + O(x\exp(-c\sqrt{\log x})). 
$$ 
Further, there exists a constant $B$ such that 
$$ 
\sum_{p\le x} \frac{a(p)^2}{p} = \log \log x + B+ O\Big(\frac{1}{\log x}\Big). 
$$ 
\end{lemma}

 Let $X$ be large, and let $P$ denote the set of primes below $X^{1/(\log \log X)^2}$ with $p\nmid N_0$.  
Let $d \in {\mathcal E}(\kappa,a)$ with $X \le |d|\le 2X$, and define 
$$ 
{\mathcal P}(d) = \sum_{p\in P} \frac{a(p)}{\sqrt{p}} \chi_d(p).
$$

 \begin{proposition} \label{Prop7} Let $k$ be a given non-negative integer. Then, for large $X$ and any $v\le X^{\frac 12-\epsilon}$,  
$$
 \sum_{\substack{d \in {\mathcal E}(\kappa,a) \\ v|d}} {\mathcal P}(d)^k \Phi\Big(\frac{\kappa d}{X} \Big) 
 =  \Big( \sum_{ \substack{d\in {\mathcal E}(\kappa, a)\\ v|d}} \Phi \Big(\frac{\kappa d}{X}\Big) \Big) (\log \log X)^{\frac{k}{2} }
 (M_k +o(1)), 
 $$ 
 where $M_k$ denotes the $k$-th Gaussian moment: 
 $$ 
 M_k = \frac{1}{\sqrt{2\pi}} \int_{-\infty}^{\infty} x^k e^{-\frac{x^2}{2} } dx  = 
 \begin{cases} 
 0&\text{ if  } k \text{ is odd}\\ 
 \frac{k!}{2^{k/2}(k/2)!} &\text{ if } k \text{ is even.}
 \end{cases}
 $$
\end{proposition}
\begin{proof}  Let $P_v$ denote the set of primes in $P$ that do not divide $v$.  Expanding ${\mathcal P}(d)^k$, we obtain 
 \begin{equation} 
 \label{4.1}
\sum_{\substack{d\in {\mathcal E}(\kappa,a) \\ v|d}}
\mathcal{P}(d)^k \Phi \Big ( \frac{\kappa d}{X} \Big )= \sum_{p_1 \in P_v} \ldots \sum_{p_k \in P_v}
\frac{a(p_1)\ldots a(p_k)}{\sqrt{p_1 \ldots p_k}} 
\sum_{\substack { d \in {\mathcal E}(\kappa,a) \\ v|d}} \chi_d(p_1 \ldots p_k) \Phi \Big ( \frac{\kappa d}{X}
\Big ).
\end{equation} 
Now we use Proposition \ref{PropDirpoly}.   If $p_1\cdots p_k$ is not a perfect square 
(which is always the case when $k$ is odd) then the sum over $d$ above is $O(X^{\frac 12+\epsilon} 
(p_1\cdots p_k)^{\frac 12})$, and the contribution of these remainder terms to \eqref{4.1} is $O(X^{\frac 12+\epsilon})$.  
This proves the proposition in the case when $k$ is odd.  

When $k$ is even, we have a main term arising from the case $p_1\cdots p_k =\square$.  
 This term contributes 
\begin{equation} 
\label{4.2}
\widehat{\Phi}(0) \frac{X}{vN_0} \prod_{p \nmid N_0} \Big ( 1 - \frac{1}{p^2}
\Big ) \prod_{p|v} \Big(1+\frac1p\Big)^{-1} 
\sum_{\substack{p_1, \ldots, p_k \in P_v \\ p_1 \ldots p_k = \square}}
\frac{a(p_1)\ldots a(p_k)}{\sqrt{p_1 \cdots p_k}} \prod_{p| p_1\cdots p_k} \Big(1+ \frac{1}{p}\Big)^{-1}.
\end{equation} 
 Suppose $q_1<q_2<\ldots<q_s$ are the distinct primes in $p_1$, $\ldots$, $p_k$.  Then each $q_j$ appears 
 an even number (say $a_j\ge 2$) of times among the $p_j$, and thus $s\le k/2$.  The terms with $s<k/2$ contribute an 
 amount 
 $$
 \ll \frac{X}{v} \Big(\sum_{p\in P_v} \frac{a(p)^2}{p} \Big)^{s} \ll X(\log \log X)^{\frac{k}2 -1},
 $$ 
 which is an acceptable error term.  When $s=k/2$, all the $a_j$ must equal $2$, and so these terms contribute 
 $$ 
 \widehat{\Phi}(0)  \frac{X}{vN_0} \prod_{p \nmid N_0} \Big ( 1 - \frac{1}{p^2}
\Big ) \prod_{p|v} \Big(1+\frac1p\Big)^{-1}  \frac{k!}{2^{k/2} (k/2)!} \sum_{\substack{{q_1,\ldots, q_{k/2} \in P_v} \\ {q_j 
\text{ distinct }}}} \frac{a(q_1)^2 \cdots a(q_{k/2})^2}{(q_1+1)\cdots (q_{k/2}+1)}.
$$ 
 Appealing to Lemma \ref{RS}, this establishes our proposition.  
%
\end{proof}

Since the normal distribution is determined by its moments, by taking $\Phi$ to approximate the 
characteristic function of $[1,2]$, summing over dyadic blocks and then over all possibilities for $\kappa$ and $a$, 
we find from Proposition \ref{Prop7} (with $v=1$) 
that for any fixed $V\in {\Bbb R}$ and as $X\to \infty$ 
\begin{align} 
\label{PdG} 
\Big| \Big\{ d\in {\mathcal E} , 20< |d|\le X: \ \frac{{\mathcal P}(d) }{\sqrt{\log \log X}} \ge V \Big\} \Big| 
\sim |\{&d\in {\mathcal E}, 20<|d|\le X\}| \notag\\ 
&\times \Big( \frac{1}{\sqrt{2\pi}}  \int_{V}^{\infty} e^{-\frac{x^2}{2}} dx + o(1)\Big). 
\end{align} 

Now if $d\in {\mathcal E}$ with $X/\log X <|d|\le X$ 
satisfies $\log L(\tfrac 12,E_d) + \frac 12\log \log X \ge V\sqrt{\log \log X}$ then we must have one of the 
following three cases: 
(1) ${\mathcal P}(d) \ge (V-\epsilon) \sqrt{\log \log X}$, or 
(2) ${\mathcal P}(d) \le -\log \log X$, or  
(3)  $-\log \log X \le {\mathcal P}(d)\le (V-\epsilon)\sqrt{\log \log X}$ but $L(\frac 12, E_d) (\log X)^{\frac 12} \exp(-{\mathcal P}(d)) \ge \exp(\epsilon \sqrt{\log \log X})$.  

From \eqref{PdG} we 
already have a satisfactory estimate for the frequency with which the first case happens.  
Next using Proposition \ref{Prop7} with $v=1$ and $k=2$ we see that case (2) appears with frequency $o(X)$.  
Finally consider case 3.  Put $\ell = 20 \lfloor \log \log X\rfloor$ so that $\ell$ is an even integer with $\ell \ge e^2 |{\mathcal P}(d)|$.   By Lemma 1 we must have $L(\frac 12,E_d) (\log X)^{\frac 12} E_\ell (-{\mathcal P}(d)) \gg \exp(\epsilon \sqrt{\log \log X})$.  
Now, a small modification of Proposition 5 shows that 
\begin{equation}
\label{LPbound} 
\sum_{\substack { d \in {\mathcal E}(\kappa,a) \\ X/\log X \le |d| \le X }}   L(\tfrac 12, E_d) (\log X)^{\frac 12} 
E_{\ell}(-\mathcal{P}(d)) 
\ll X \log \log X, 
\end{equation}
and so case (3) also occurs with frequency $o(X)$.  This completes our proof.

 \section{Proof of Theorem \ref{thm3}} 
 
 \noindent Recall that the elliptic curve $E$ is given in Weierstrass form by $y^2 =f(x)$ for a monic cubic polynomial 
 $f$ with integer coefficients, and that $K$ is the splitting field of $f$ over ${\Bbb Q}$.  Let $c(p)$ denote $1$ plus the 
 number of solutions to $f(x) \equiv 0 \pmod p$, so that $c(p)=1$, $2$, or $4$.  The Tamagawa number $\text{Tam}(E_d) = 
 \prod_{p} T_p(d)$ may be calculated using Tate's algorithm (see \cite{Rubin}).  Primes dividing the discriminant of $f$ 
 make a bounded contribution, and for a prime not dividing the discriminant the factor $T_p(d)$ equals $c(p)$ if $p|d$, and $T_p(d)=1$ otherwise.  
 
 \begin{lemma} 
 \label{Tam}  In the notation of Conjecture 1, we have 
 $$ 
 \sum_{p\le x} \frac{\log c(p)}{p} = \Big( -\mu(E)- \frac 12\Big) \log \log x  +O(1), 
 $$ 
 and 
 $$ 
 \sum_{p\le x} \frac{(\log c(p))^2}{p} = (\sigma(E)^2-1) \log \log x + O(1). 
 $$ 
 \end{lemma}
 \begin{proof}  Let us consider the case when $[K:{\Bbb Q}]=6$, so that the extension has Galois group $S_3$.  
 The Chebotarev density theorem gives that $c(p)=4$ for a set of primes of density $1/6$, $c(p)=2$ on a set of primes of density $1/2$, and $c(p)=1$ on a set of primes of density $1/3$.  This proves the lemma in this case, and the other 
 cases are similar.
 \end{proof}   
 
 Let $X$ be large and define ${\mathcal P}(d)$ as in Section 4.  Further define, for primes $p\nmid N_0$,  
 $$ 
 C_p(d) = \begin{cases} 
 \frac{p}{p+1} \log c(p) &\text{ if } p|d\\ 
 -\frac{1}{p+1} \log c(p) &\text{ if }p\nmid d. 
 \end{cases}
 $$ 
Put $z= X^{1/(\log \log X)^2}$ and set 
 \begin{equation} 
 \label{E5.1} 
 {\mathcal C}(d) = \sum_{\log X \le p \le z} C_p(d) = \sum_{\log X \le p \le z} \Big( \log T_p(d) - \frac{\log c(p)}{p+1}\Big). 
 \end{equation} 
 Since the real period $\Omega(E_d)$ is $\asymp 1/\sqrt{|d|}$, and $|E_d({\Bbb Q})_{\text{tors}}|$ is 
 bounded, in order to prove Theorem 3, it suffices to estimate 
 \begin{equation} 
 \label{E5.2} 
 \Big| \Big\{ d\in {\mathcal E}, \frac{X}{\log X}\le |d|\le X: \ 
 \frac{\log L(\tfrac 12,E_d) - \sum_{p|d} \log c(p) -\mu(E) \log \log X}{\sqrt{\sigma(E)^2\log \log X}} \ge V \Big\} \Big|.
 \end{equation} 
 If $d$ is a discriminant counted in \eqref{E5.2} then one of the following four cases must happen: (1) ${\mathcal P(d)} - 
 {\mathcal C}(d) \ge (V-\epsilon) \sqrt{\sigma(E)^2 \log \log X}$, or (2) ${\mathcal P}(d) \le -\log \log X$, or (3) $-\log \log X \le {\mathcal P}(d) \le (V-\epsilon) \sqrt{\log \log X}$ but $L(\tfrac 12, E_d) (\log X)^{\frac 12} \exp(-{\mathcal P}(d)) \ge 
 \exp(\epsilon \sqrt{\log \log X})$, or (4) $|\log \text{Tam}(E_d) +(\mu(E)+\frac 12) \log \log X - {\mathcal C}(d)| \ge \frac{\epsilon}{10} \sqrt{\log \log X}$.  
 
 From our work in Section 4, we know that cases 2 and 3 occur for at most $o(X)$ discriminants $d$.   Now consider 
 case 4.  By Lemma \ref{Tam} 
 $$ 
 |\log \text{Tam}(E_d) + (\mu(E)+\tfrac 12) \log \log X -{\mathcal C}(d)| 
 = \sum_{\substack{ p|d \\ p<\log X}} \log c(p)  +\sum_{\substack {p|d \\ p>z}}\log c(p) + O(\log \log \log X). 
 $$ 
 Summing the above over all $d\in {\mathcal E}$ with $X/\log X\le |d|\le X$ we get 
 \begin{align*}
 \sum_{\substack{ d\in {\mathcal E} \\ X/\log X\le |d|\le X}} &|\log \text{Tam}(E_d) + (\mu(E)+\tfrac 12) \log \log X -{\mathcal C}(d)|  
\\
& \ll X\log\log  \log X + X\sum_{p<\log X} \frac{\log c(p)}{p} +X \sum_{X\ge p>z} \frac{\log c(p)}{p} \ll X\log \log \log X. 
 \end{align*}
 Therefore case 4 also occurs with frequency $o(X)$.  It remains lastly to estimate the occurrence  of case 1, 
 which we achieve by computing the moments of ${\mathcal P}(d) -{\mathcal C}(d)$, and showing that these approximate 
 the moments of a normal distribution with mean zero and variance $\sigma(E)^2\log \log X$; our work here follows the 
 argument in \cite{GrSo}.   
 Since, as noted already, the Gaussian is determined by its moments, this completes the proof of Theorem 3.

 \begin{proposition} \label{Prop8}  Let $k$ be a given non-negative integer.  Then for large $X$ we have 
 $$ 
 \sum_{d\in {\mathcal E}(\kappa, a)} ({\mathcal P}(d) -{\mathcal C}(d))^k \Phi\Big(\frac{\kappa d}{X}\Big) = 
 (\sigma(E)^2 \log \log X)^{\frac k2} (M_k+o(1)) \sum_{d\in {\mathcal E}(\kappa, a)} \Phi\Big(\frac{\kappa d}{X}\Big).
 $$ 
 \end{proposition} 
\begin{proof}  Expanding out our sum, we must evaluate 
\begin{equation} 
\label{E5.3}
\sum_{j=0}^{k} \binom{k}{j} (-1)^j \sum_{(\log X)\le p_1, \ldots, p_j \le z} \sum_{d\in {\mathcal E}(\kappa, a)} 
C_{p_1}(d)\cdots C_{p_j}(d) {\mathcal P}(d)^{k-j} \Phi\Big(\frac{\kappa d}{X}\Big). 
\end{equation}  
Suppose that $q_1 < q_2 <\ldots <q_{\ell}$ are the distinct primes appearing in $p_1$, $\ldots$, $p_{j}$ 
and that $q_i$ appears with multiplicity $a_i$.  For such a choice of $p_1$, $\ldots$, $p_j$, note that 
$$ 
C_{p_1}(d)\cdots C_{p_j}(d) = \prod_{i=1}^{\ell} C_{q_i}(d)^{a_i} 
= \sum_{v|(d,q_1\cdots q_{\ell})} \Big(\sum_{rs=v} \mu(r) \prod_{i=1}^{\ell} C_{q_i}(s)^{a_i} \Big).
$$
Therefore the inner sum over $d$ in \eqref{E5.3} equals 
\begin{equation} 
\label{E5.4}
\sum_{v|q_1\cdots q_{\ell}} \Big(\sum_{rs=v}\mu(r) \prod_{i=1}^{\ell} C_{q_i}(s)^{a_i} \Big) 
\sum_{\substack { d\in {\mathcal E}(\kappa,a) \\ v|d }} {\mathcal P}(d)^{k-j} \Phi\Big(\frac{\kappa d}{X}\Big). 
\end{equation} 
Using our work from Section 4, and in particular \eqref{4.2} there, we see that the sum over $d$ above is 
$$ 
{\check \Phi}(0)\frac{X}{N_0} \prod_{p\nmid N_0} \Big(1-\frac{1}{p^2} \Big) 
\prod_{p|v} \Big(\frac{1}{p+1}\Big) \sum_{\substack{p_{j+1},\ldots, p_k \in P_{v} \\ p_{j+1}\cdots p_k=\square}} 
\frac{a(p_{j+1})\cdots a(p_k)}{\sqrt{p_{j+1}\cdots p_k }} \prod_{p|p_{j+1}\cdots p_k} \Big(1+\frac{1}{p}\Big)^{-1} + O(X^{\frac 12+\epsilon} ). 
$$ 
Above, as in Section 4, $P_v$ denotes the set of primes below $z$ that do not divide $N_0$ or $v$.  But the contribution 
of terms above with some $p_i$ dividing $v$ is easily seen to be $O(X(\log \log X)^k/(v\log X))$, since all the prime factors of $v$ are larger than $\log X$.   Thus removing the restriction that $p_{i}$ do not divide $v$ (for $j+1\le i\le k$), we see that 
the quantity in \eqref{E5.4} is,  up to an error $O(X(\log \log X)^k/(v\log X))$, 
\begin{equation} 
\label{E5.5} 
{\check \Phi}(0)\frac{X}{N_0} \prod_{p\nmid N_0} \Big(1-\frac{1}{p^2} \Big) \sum_{\substack{p_{j+1},\ldots, p_k \in P \\ p_{j+1}\cdots p_k=\square}} \frac{a(p_{j+1})\cdots a(p_k)}{\sqrt{p_{j+1}\cdots p_k }} \prod_{p|p_{j+1}\cdots p_k} \Big(1+\frac{1}{p}\Big)^{-1}  G(q_1^{a_1}\cdots q_{\ell}^{a_\ell}) , 
\end{equation} 
where 
$$ 
G(q_1^{a_1}\cdots q_{\ell}^{a_\ell}) = 
 \sum_{v|q_1\cdots q_{\ell}} \Big( \sum_{rs=v} \mu(r) \prod_{i=1}^{\ell} C_{q_i}(s)^{a_i}\Big) \prod_{p|v}\Big(\frac{1}{p+1}\Big).   
 $$
 Now it is easy to see that $G$ is a multiplicative function and that 
 $$
 G(p^a) = (\log c(p))^a \Big( \frac{1}{p+1} \Big(1-\frac{1}{p+1}\Big)^a + \frac{p}{p+1} \Big(-\frac{1}{p+1}\Big)^a \Big). 
 $$ 
 Hence $G(q_1^{a_1} \cdots q_\ell^{a_\ell})$ is non-zero only if all the $a_i$ are at least $2$, and $G(p^a) \ll (\log c(p))^a/p$ 
 for all $a\ge 2$.  
 
 We now use this evaluation of the inner sum over $d$ in \eqref{E5.3}, and then perform the sum over $p_1$, $\ldots$, $p_j$.  
 First note that the error term incurred above leads to a total error of at most $O(X (\log \log X)^{2k}/(\log X))$, which is 
 acceptable for the proposition.  Now let us simplify the main term that arose above.  Given $q_1 < \ldots < q_{\ell}$ 
 and $a_i \ge 2$ with $\sum a_i= j$, the number of choices for $p_1$, $\ldots$, $p_j$ is $j!/(a_1!\cdots a_{\ell}!)$.   Thus 
 the main term is 
 \begin{align} 
 \label{E5.6}
 {\check \Phi}(0) \frac{X}{N_0} \prod_{p\nmid N_0} \Big(1-\frac 1{p^2}\Big) &
 \sum_{j=0}^{k} \binom{k}{j} (-1)^j \sum_{\substack{p_{j+1}, \ldots, p_k \in P \\ p_{j+1}\cdots p_k=\square } } 
 \frac{a(p_{j+1})\cdots a(p_k)}{\sqrt{p_{j+1}\cdots p_k} } \prod_{p|p_{j+1}\cdots p_k} \Big(1+\frac 1p\Big)^{-1} \nonumber \\
&\times \sum_{\substack{\ell \\ a_1,\ldots a_{\ell} \ge 2 \\ \sum a_i=j}} \frac{j!}{\prod_{i} a_i!} \sum_{\substack{\log X \le q_1<\ldots < q_{\ell} < z}} G(q_1^{a_1} \cdots q_\ell^{a_\ell} ). 
 \end{align}
 
 Now if any $a_i \ge 3$, then the sum over $q_i$ above is seen to be $\ll (\log \log X)^{(j-1)/2}$, and the sum over 
 $p_{j+1}$, $\ldots$, $p_k$ contributes $\ll (\log \log X)^{(k-j)/2}$, leading to a total of $\ll X(\log \log X)^{(k-1)/2}$.  Thus 
 the effect of such terms is negligible, and we are left with the case when all $a_i=2$, so that $j=2\ell$ is even.  Since $G(p^2) 
 \sim (\log c(p))^2/p$, using Lemma 4, 
 these terms contribute 
 \begin{align*}
 \sim {\check \Phi}(0) \frac{X}{N_0} \prod_{p\nmid N_0} \Big(1-\frac{1}{p^2} \Big) &
 \sum_{\substack{j=0 \\ j \text{ even }}}^k \binom{k}{j} \frac{j!}{2^{j/2} (j/2)!} ((\sigma(E)^2-1) \log \log X)^{j/2} \\
&\times \sum_{\substack{p_{j+1}, \ldots, p_k \in P \\ p_{j+1}\cdots p_k=\square } } 
 \frac{a(p_{j+1})\cdots a(p_k)}{\sqrt{p_{j+1}\cdots p_k} } \prod_{p|p_{j+1}\cdots p_k} \Big(1+\frac 1p\Big)^{-1} .
 \end{align*}
 Now, the sum over $p_{j+1}, \ldots, p_k$ is zero unless $k-j$ is even (so that $k$ is even), and in that case, 
 arguing as in Section 4, it equals $\sim \frac{(k-j)!}{2^{(k-j)/2} ((k-j)/2)!} (\log \log X)^{(k-j)/2}$.  Using this above, we conclude the proposition.
\end{proof}  
 
  \section{Preliminary Lemmas} 

\subsection{The approximate functional equation}

\begin{lemma} \label{AFE}  For $d\in {\mathcal E}(\kappa, a)$ we have 
$$ 
L(\tfrac 12, E_d) = 2\sum_{\substack{{n=1} \\ {(n,N_0)=1}}}^{\infty} \frac{a(n)}{\sqrt{n} }\chi_d(n) W\Big( \frac{n}{|d|} \Big),  
$$ 
where for $\xi >0$ and any $c>0$ we define  
$$ 
W(\xi)  = \frac{1}{2\pi i } \int_{c-i\infty}^{c+i\infty} L_a(s+\tfrac 12) \Gamma(s) \Big(\frac{\sqrt{N}}{2\pi \xi}\Big)^s ds.
$$ 
The function $W(\xi)$ is smooth in $\xi >0$ and satisfies $W^{(k)}(\xi) \ll_k \xi^{-k} e^{-2\pi \xi/\sqrt{N}}$ for non-negative 
integers $k$ and further we have $W(\xi) = L_a(\frac 12) + O(\xi^{\frac 12-\epsilon})$ as $\xi \to 0$.  
\end{lemma} 
\begin{proof}  We begin with, for $c > \tfrac 12$,  
$$ 
I= \frac{1}{2\pi i} \int_{c-i\infty}^{c+i\infty} \Big(\frac{\sqrt{N}|d|}{2\pi}\Big)^{s} \Gamma(s+1) L(s+\tfrac 12, E_d) \frac{ds}{s}. 
$$ 
On the one hand, since 
$$ 
L(s+\tfrac 12, E_d) = L_a(s+\tfrac 12) \sum_{\substack{{n=1}\\ {(n,N_0)=1}}}^{\infty} \frac{a(n)}{\sqrt{n}} \chi_d(n),
$$ 
integrating term by term we see that  
$$ 
I = \sum_{\substack{{n=1}\\ {(n,N_0)=1}}}^{\infty} \frac{a(n)}{\sqrt{n}} \chi_d(n) W\Big(\frac{n}{|d|} \Big). 
$$ 
On the other hand, moving the line of integration in $I$ to $-c$ and using the functional equation (note that the sign is positive by assumption) we 
may see that 
$$ 
I= L(\tfrac 12,E_d) - I,
$$ 
and the stated identity follows.   Now from the definition of $W$ we see that 
$$ 
W(\xi) = \sum_{\substack{n=1 \\ p|n \implies p|N_0}}^{\infty} \frac{a(n)}{\sqrt{n}} \chi_d(n) e^{-2\pi \xi n/\sqrt{N}}.
$$ 
From this it is clear that $W$ is smooth in $\xi>0$, and the stated bound on $W^{(k)}(\xi)$ follows.  The last claim on the behavior of $W(\xi)$ as $\xi\to 0$ is obtained by moving the line of integration to Re$(s)=-\frac 12+\epsilon$, and picking up the contribution of the pole at $s=0$.  
\end{proof} 

\subsection{Poisson summation}

Here we apply the Poisson summation formula to understand real character sums.   Let $n$ be an odd integer 
and define the Gauss type sum $G_k(n)$ for any integer $k$ by 
$$ 
G_k(n) = \Big( \frac{1-i}{2} + \Big(\frac{-1}{n}\Big) \frac{1+i}{2} \Big) \sum_{a\pmod n} \Big(\frac{a}{n} \Big) e\Big(\frac{ak}{n}\Big). 
$$ 
In addition to $G_k$, it is helpful to define the closely related sum
$$ 
\tau_k(n) = \sum_{b \pmod n} \Big(\frac{b}{n}\Big) e\Big(\frac{kb}{n}\Big)= \Big(\frac{1+i}{2} + \Big(\frac{-1}{n}\Big) \frac{1-i}{2}\Big) G_k(n). 
$$ 
The Gauss type sum $G_k(n)$ has been calculated explicitly in Lemma 2.3 of \cite{So4} which we now quote. 

\begin{lemma} \label{Gauss}  If $m$ and $n$ are coprime odd integers then $G_k(mn)=G_k(m)G_k(n)$.  
If $p^{\alpha}$ is the largest power of $p$ dividing $k$ (setting $\alpha=\infty$ if $k=0$) then 
$$ 
G_k(p^{\beta}) = \begin{cases}
0 &\text{if }  \beta \le \alpha \text{ is odd},\\ 
\phi(p^{\beta}) &\text{if } \beta \le \alpha \text{ is even},\\
-p^{\alpha} &\text{if } \beta= \alpha+1 \text{ is even},\\ 
 (\frac{kp^{-\alpha}}{p}) p^{\alpha}\sqrt{p} &\text{if } \beta=\alpha+1 \text{ is odd},\\
 0 &\text{if } \beta \ge \alpha+2.
 \end{cases}
 $$ 
\end{lemma}

\begin{lemma} \label{Poisson}  Let $r \pmod q$ be a given residue class, and let $n$ be an odd natural number coprime to 
$q$.   Let $F$ be a smooth compactly supported function.  Then 
$$ 
\sum_{d \equiv r \pmod q} \Big(\frac{d}{n}\Big) F(d)= \frac{1}{qn} \Big(\frac{q}{n}\Big) 
\sum_{k} {\widehat F}\Big(\frac{k}{nq}\Big) e\Big(\frac{kr\overline{n}}{q}\Big) \tau_k(n),  
$$ 
where ${\widehat F}$ denotes the Fourier transform.  
\end{lemma}
\begin{proof}  The desired sum is 
$$ 
\sum_{b\pmod n} \Big(\frac{b}{n}\Big) \sum_{\substack{{d\equiv r\pmod q}\\ {d\equiv b\pmod n}}} F( {d} ). 
$$ 
Since $q$ and $n$ are coprime, the congruence conditions above may be expressed as $d\equiv bq\overline{q} +rn\overline{n} 
\pmod{qn}$ where $q\overline{q}\equiv 1\pmod n$ and $n\overline n \equiv 1\pmod q$.  Thus, using Poisson summation the inner sum over $d$ equals 
$$ 
\sum_{d\equiv bq\overline{q} +rn\overline{n} 
\pmod{qn}} F(d) = \frac{1}{qn} \sum_{k} {\widehat F}\Big( \frac{k}{nq} \Big) e\Big( \frac{kb\overline{q}}{n} + 
\frac{k r\overline{n}}{q}\Big). 
$$
Bringing back the sum over $b$ we conclude that the desired sum equals 
$$ 
\frac{1}{qn} \sum_{k} { \widehat F}\Big(\frac{k}{nq} \Big) e\Big(\frac{kr\overline{n}}{q} \Big) \sum_{b\pmod n} \Big(\frac{b}{n} \Big) e\Big(\frac{kb\overline{q}}{n}\Big)= \frac{1}{qn} \Big(\frac{q}{n}\Big) \sum_{k} {\widehat F}\Big(\frac{k}{nq} \Big) e\Big(\frac{kr\overline{n}}{q}\Big) \tau_{k}(n). 
$$ 
\end{proof}

\section{Proof of Proposition \ref{PropDirpoly}} 

\noindent Since $v$ is square-free and coprime to $N_0$, note that, (for $d$ coprime to $N_0$, and $d$ a multiple of $v$) 
\begin{equation} 
\label{S21} 
\sum_{\beta|(v,d/v)} \mu(\beta) \sum_{\substack{{(\alpha, vN_0)=1}\\ {\alpha^2 | d/v}}} \mu(\alpha) = \begin{cases} 
1 &\text{if  } d \text{ is a square-free multiple of } v\\ 
0 &\text{otherwise}. 
\end{cases}
\end{equation}  
Thus, writing $d=kv\beta\alpha^2$, we obtain 
\begin{equation} 
\label{7.2}
\sum_{\substack{d \in {\mathcal E}(\kappa,a)  \\ v | d}} 
\chi_d(n) \Phi \Big ( \frac{\kappa d}{X} \Big) 
= 
\sum_{\beta | v} \sum_{(\alpha, v N_0) = 1} \mu(\beta) \mu(\alpha)
\Big( \frac{v \beta \alpha^2}{n} \Big) \sum_{k \equiv a \overline{v \beta \alpha^2} \pmod{N_0}} \Big( \frac{k}{n}\Big) \Phi \Big(
\frac{\kappa k v \beta \alpha^2}{X} \Big).
\end{equation}  

Put $A=X^{\frac 12 -\epsilon}/(v\sqrt{n}).$  Estimating the sum over $k$ trivially, 
the terms in \eqref{7.2} with $\alpha >A$ contribute 
\begin{equation} \label{7.3} 
\ll \sum_{\beta |v} \sum_{\alpha >A} \frac{X}{v\beta \alpha^2} \ll \frac{Xv^{\epsilon}}{vA} \ll X^{\frac 12+\epsilon} \sqrt{n}. 
\end{equation} 

For the terms with $\alpha \le A$, we use the Poisson summation as stated in Lemma \ref{Poisson}, which gives  
\begin{equation} 
\label{7.4}
\sum_{k \equiv a \overline{v \beta \alpha^2} \pmod{N_0}}
\Big ( \frac{k}{n} \Big ) \Phi \Big ( \frac{\kappa k v \beta \alpha^2}{X}
\Big ) = \frac{X}{n N_0 v \beta \alpha^2} \Big ( \frac{\kappa N_0}{n} \Big )
\sum_{\ell} \widehat{\Phi} \Big ( \frac{X \ell}{n v \beta \alpha^2 N_0} \Big )
e \Big ( \frac{\ell a \overline{v \beta \alpha^2 n}}{N_0} \Big ) \tau_{\ell}(n).
\end{equation}
Since $X/(nv\beta \alpha^2 N_0) \ge X^{\epsilon}/N_0$ and ${\hat \Phi}(\xi) \ll_K |\xi|^{-K}$ for any $K>0$, we 
see that the terms with $\ell \neq 0$ above contribute (using the trivial bound $|\tau_{\ell}(n)|\le n$) 
an amount $\ll X^{-1}$ say.   If $n$ is not a perfect square, then the term $\ell =0$ above vanishes, and 
we conclude that the quantity in \eqref{7.4} is $\ll X^{-1}$.  Using this in \eqref{7.2} we see that the terms with $\alpha\le A$ 
contribute $\ll X^{-1+\epsilon} A$ in this case.   Thus the proposition follows in the case when $n$ is not a 
perfect square.

When $n$ is a perfect square, the term $\ell=0$ makes a contribution of ${\hat \Phi}(0)\frac{\phi(n)}{n} \frac{X}{vN_0\beta\alpha^2}$ in 
\eqref{7.4}.   Thus the terms with $\alpha \le A$ contribute to \eqref{7.2} 
\begin{align*}
&\widehat{\Phi}(0) \frac{X}{vN_0} \frac{\phi(n)}{n}
\sum_{\beta | v} \frac{\mu(\beta)}{\beta} \sum_{\substack{(\alpha, v n N_0) = 1 \\ \alpha \le A}} \frac{\mu(\alpha)}{
\alpha^2} + O(AX^{-1+\epsilon})\\
=& 
\widehat{\Phi}(0) \frac{X}{vN_0} \frac{\phi(n)}{n} \sum_{\beta | v} \frac{\mu(\beta)}{\beta} \sum_{\substack{(\alpha, v n N_0) = 1 }} \frac{\mu(\alpha)}{
\alpha^2} + O(X^{\frac 12+\epsilon} \sqrt{n}) .
\end{align*}
It is easy to check that the main term above equals  
$$
\widehat{\Phi}(0) \frac{X}{v N_0} \prod_{p \nmid N_0} \Big ( 1 - \frac{1}{p^2} \Big )
\prod_{p | nv} \Big ( 1 + \frac{1}{p} \Big )^{-1}, 
$$ 
and so the proposition follows in the case when $n$ is a perfect square, completing our proof.

\section{Proof of Proposition \ref{Prop3}} 
\label{SProp3}

\noindent Let ${\tilde a}(n)$ denote the completely multiplicative function defined on primes $p$ by ${\tilde a}(p) = a(p)$.  Let $w(n)$ 
be the multiplicative function defined by $w(p^{\alpha}) = \alpha!$ for prime powers $p^{\alpha}$.   For $1\le j\le R$ we may write 
\begin{equation} 
\label{5.1} 
{\mathcal B}_j(d) = \sum_{n_j} \frac{{\tilde a}(n_j)}{\sqrt{n_j}} \frac{k^{\Omega(n_j)}}{w(n_j)}  b_j(n_j) \chi_d(n_j),
\end{equation} 
where $\Omega(n_j)$ denotes the number of prime factors of $n_j$ (counted with multiplicity), and $b_j(n_j)=1$ if 
$n_j$ is composed of at most $\ell_j$ primes, all from the interval $P_j$, and $b_j(n_j)$ is zero otherwise.    In particular, note that 
$b_j(n_j)=0$ unless $n_j< (X^{1/\ell_j^2})^{\ell_j}=X^{1/\ell_j}$, so that ${\mathcal B}_j(d)$ is a short Dirichlet polynomial.  
Write also 
\begin{equation} 
\label{5.2}
\frac{1}{\ell_j!} {\mathcal P}_j(d)^{\ell_j} = \sum_{n_j} \frac{{\tilde a}(n_j)}{w(n_j)\sqrt{n_j}} p_j(n_j) \chi_d(n_j), 
\end{equation} 
where $p_j(n_j)=1$ if $n_j$ is composed of exactly $\ell_j$ primes (counted with multiplicity) all from the interval $P_j$, 
and $p_j(n_j)=0$ otherwise.  This too is a short Dirichlet polynomial supported only on $n_j \le X^{1/\ell_j}$.  
Thus note that $\prod_{j=1}^{R} {\mathcal B}_j(d)$ and $\prod_{j=1}^{r} {\mathcal B}_j(d) {\mathcal P}_{r+1}^{\ell_{r+1}}$ 
are all short Dirichlet polynomials, of length at most $X^{1/\ell_1+ \ldots +1/\ell_R} < X^{1/1000}$.   

Let $0\le r\le R-1$ and consider one of the terms $\prod_{j=1}^{r} {\mathcal B}_j(d) {\mathcal P}_{r+1}^{\ell_{r+1}}$ that arises in our proposition.   We expand this term using \eqref{5.1} and \eqref{5.2}, and appeal to Proposition \ref{PropDirpoly} (with $v=1$ there).   Since the Dirichlet polynomials ${\mathcal B}_j$ and ${\mathcal P_j}^{\ell_j}$ are short, the error terms arising from Proposition \ref{PropDirpoly} 
contribute a negligible amount.  We are thus left with the main term, which is 
\begin{align*}
 \frac{X}{N_0} {\check \Phi}(0) \prod_{p\nmid N_0} \Big(1-\frac{1}{p^2} \Big) 
 &\prod_{j=1}^{r} \Big( \sum_{n_j =\square} \frac{{\tilde a}(n_j)}{\sqrt{n_j}} \frac{k^{\Omega(n_j)}}{w(n_j)} \prod_{p|n_j} \Big(1+\frac 1p\Big)^{-1} b_j(n_j)\Big)\\
 &\times  \Big( \ell_{r+1}! \sum_{n_{r+1}=\square} \frac{{\tilde a}(n_{r+1})}{w(n_{r+1}) \sqrt{n_{r+1}}} 
\prod_{p|n_{r+1}} \Big(1+\frac 1p\Big)^{-1} p_{r+1}(n_{r+1}) \Big). 
\end{align*}

Now, since all the terms involved are non-negative,  
$$ 
\sum_{n_j=\square} \frac{{\tilde a}(n_j)}{\sqrt{n_j}} \frac{k^{\Omega(n_j)}}{w(n_j)} \prod_{p|n_j} \Big(1+\frac{1}{p}\Big)^{-1} b_j(n_j) 
\le \prod_{p\in P_j} \Big( \sum_{t=0}^{\infty} \frac{a(p)^{2t}}{p^t} \frac{k^{2t}}{(2t)!} \Big) 
\ll \exp\Big( \frac{k^2}{2} \sum_{p\in P_j} \frac{a(p)^2}{p}\Big). 
$$ 
Similarly we find that 
$$ 
\sum_{{n_{r+1}=\square} }\frac{{\tilde a}(n_{r+1})}{w(n_{r+1}) \sqrt{n_{r+1}}} 
\prod_{p|n_{r+1}} \Big(1+\frac 1p \Big)^{-1} p_{r+1}(n_{r+1}) \le \frac{1}{(\ell_{r+1}/2)!} \Big( \sum_{p\in P_{r+1}} \frac{a(p)^2}{p} \Big)^{\ell_{r+1}/2}.
$$ 

Putting all these observations together, we find that 
\begin{align*}
\sum_{d \in {\mathcal E}(\kappa,a)} \prod_{j=1}^{r} {\mathcal B}_j(d) \Big( \frac{e^2 {\mathcal P}_{r+1}}{\ell_{r+1}}\Big)^{\ell_{r+1}} \Phi\Big(\frac{\kappa d}X\Big)
&\ll X \exp\Big(\frac{k^2}{2} \sum_{j=1}^{r} \sum_{p\in P_j} \frac{a(p)^2}{p} \Big)\\
&\times \Big( \Big(\frac{e^2}{\ell_{r+1}}\Big)^{\ell_{r+1}} \frac{\ell_{r+1}!}{(\ell_{r+1}/2)!} 
\Big( \sum_{p\in P_{r+1}} \frac{a(p)^2}{p} \Big)^{\ell_{r+1}/2}\Big). 
\end{align*}
Using Stirling's formula, Lemma \ref{RS}, and that $\ell_{r+1} \ge 10^4$ we find that the above 
is 
$$ 
\ll X e^{-\ell_{r+1}} (\log X)^{\frac{k^2}{2}}. 
$$ 
Arguing in the same way, we obtain 
$$ 
\sum_{d\in {\mathcal E}(\kappa, a)} \prod_{j=1}^{R} {\mathcal B}_j(d ) \Phi\Big(\frac{\kappa d}{X}\Big)  \ll X (\log x)^{\frac{k^2}{2}}. 
$$ 
Summing all these bounds, the proposition follows.

\section{Proof of Proposition \ref{Prop4}} 
\label{SProp4}

\noindent Let ${\tilde a}(n)$, $w(n)$, $b_j(n)$ and $p_j(n)$ be defined as in Section \ref{SProp3}.  Then 
for $1\le j\le R$ we may write
$$ 
{\mathcal A}_j(d) = \sum_{n_j} \frac{{\tilde a}(n_j)}{\sqrt{n_j}} \frac{(k-1)^{\Omega(n_j)}}{w(n_j)} b_j(n_j) \chi_d(n_j).
$$ 
We shall also use the expression \eqref{5.2}.  Thus, as in Section \ref{SProp3},  $\prod_{j=1}^{R} {\mathcal A}_j(d)$ and 
$\prod_{j=1}^r {\mathcal A}_j(d) {\mathcal P}_{{r+1}}^{\ell_{r+1}}$ are short Dirichlet polynomials of 
length at most $X^{1/1000}$.

Let $0\le r\le R-1$ and consider one of the terms $\prod_{j=1}^{r} {\mathcal A}_j(d) {\mathcal P}_{r+1}^{\ell_{r+1}}$ that 
arises in our proposition.   We expand this term into its Dirichlet series, and appeal to Proposition 2 (with $v=1$ there).  
The error terms are negligible and we are left once again with the main term, which is 
\begin{align}
\label{6.1}
{C(a,E){\check \Phi}(0)} X &\prod_{j=1}^{r} \Big( \sum_{n_j } \frac{{\tilde a}(n_j)}{\sqrt{n_j}} \frac{a(n_{j1})}{\sqrt{n_{j1}}} 
\frac{(k-1)^{\Omega(n_j)}}{w(n_j) } b_j(n_j) g(n_{j}) \Big) \nonumber \\
&\times \Big( \ell_{r+1}! \sum_{n_{r+1}} \frac{{\tilde a}(n_{r+1}) a(n_{(r+1)1})}{\sqrt{n_{r+1}n_{(r+1)1}}} \frac{g(n_{r+1}) p_{r+1}(n_{r+1})}{w(n_{r+1})}\Big).
\end{align}
Here $C(a,E)$ is a constant, depending only on $a$ and $E$; we write $n_j=n_{j1}n_{j2}^2$ with $n_{j1}$ square-free; 
and $g$ is the multiplicative function of Proposition 2.   

Consider one of the terms with $1\le j\le r$ in \eqref{6.1} above.   The factor $b_{j}(n_j)$ constrains $n_j$ to have 
all prime factors in $P_j$, and also restricts $\Omega(n_j)$ to be at most $\ell_j$.  If we ignore the restriction on $\Omega(n_j)$, 
the result would be given by the Euler product  
\begin{equation} 
\label{6.2} 
\prod_{p\in P_j} \Big( \sum_{j=0}^{\infty} \frac{a(p)^{2j}}{p^j} \frac{(k-1)^{2j}}{(2j)!} g(p^{2j}) + \sum_{j=0}^{\infty} \frac{a(p)^{2j+2}}{p^{j+1}} 
\frac{(k-1)^{2j+1}}{(2j+1)!} g(p^{2j+1})\Big);  
\end{equation} 
the first sum above counts terms where $n_{j}$ is divisible by an even power of $p$ (so that $n_{j1}$ is not a multiple of $p$), 
and the second sum counts those terms with an odd power of $p$ dividing $n_{j}$ (so that $n_{j1}$ is divisible by $p$).   
The error in replacing the quantity in \eqref{6.1} by the Euler product of \eqref{6.2} comes from terms with $\Omega(n_j)>\ell_j$.  We estimate this error using ``Rankin's trick."   Since $2^{\Omega(n_j)-\ell_j}\ge 1$ if $\Omega(n_j) > \ell_j$,  the error in passing from \eqref{6.1} to \eqref{6.2} is at most
$$ 
\sum_{n_j} \frac{|{\tilde a}(n_j)|}{\sqrt{n_j}} \frac{|a(n_{j1})|}{\sqrt{n_{j1}}} 
\frac{|k-1|^{\Omega(n_j)}}{w(n_j) } 2^{\Omega(n_j)-\ell_j} g(n_{j}), 
$$
which is 
\begin{align*}
&\le 2^{-\ell_j} 
\prod_{p\in P_j} \Big( \sum_{j=0}^{\infty} \frac{a(p)^{2j}}{p^j} \frac{(k-1)^{2j}2^{2j}}{(2j)!} g(p^{2j}) 
+ \sum_{j=0}^{\infty} \frac{a(p)^{2j+2}}{p^{j+1}} \frac{|k-1|^{2j+1} 2^{2j+1}}{(2j+1)!} g(p^{2j+1})\Big) \\
&\ll 2^{-\ell_j} 
\exp\Big( 4\sum_{p\in P_j} \frac{a(p)^2}{p} \Big). 
\end{align*} 
Since the quantity in \eqref{6.2} is $\gg \exp(\frac{k^2-1}{2} \sum_{p\in P_j} \frac{a(p)^2}{p})$ and,
using Lemma \ref{RS} and the definition of $\ell_j$, we may check that $\ell_j \ge 10 \sum_{p\in P_j} a(p)^2/p$, we conclude that 
\begin{align}
\label{6.3}
\sum_{n_j } \frac{{\tilde a}(n_j)}{\sqrt{n_j}} \frac{a(n_{j1})}{\sqrt{n_{j1}}} 
&\frac{(k-1)^{\Omega(n_j)}}{w(n_j) } b_j(n_j) g(n_{j}) 
= (1+O(2^{-\ell_j/2})) \nonumber\\ 
&\times \prod_{p\in P_j} \Big( \sum_{j=0}^{\infty} \frac{a(p)^{2j}}{p^j} \frac{(k-1)^{2j}}{(2j)!} g(p^{2j})+ \sum_{j=0}^{\infty} \frac{a(p)^{2j+2}}{p^{j+1}} 
\frac{(k-1)^{2j+1}}{(2j+1)!} g(p^{2j+1})\Big). 
\end{align}

Now consider the contribution of the $n_{r+1}$ terms in \eqref{6.1}.  Note that the 
terms here satisfy $\Omega(n_{r+1}) =\ell_{r+1}$, and we estimate these using Rankin's method again.   Thus, the 
contribution of the $n_{r+1}$ terms is 
$$ 
\le \ell_{r+1}! 10^{-\ell_{r+1}} \prod_{p\in P_{r+1}} 
\Big( \sum_{j=0}^{\infty} \sum_{j=0}^{\infty}\frac{a(p)^{2j} 10^{2j}}{p^j (2j)!} g(p^{2j})+ \sum_{j=1}^{\infty} 
\frac{a(p)^{2j+2} 10^{2j+1} }{p^{j+1} (2j+1)!} g(p^{2j+1})\Big). 
$$ 
Since $\ell_{r+1}! \le \ell_{r+1} (\ell_{r+1}/e)^{\ell_{r+1}}$, the above is, using Lemma \ref{RS} and the definition of $\ell_{j}$,
$$ 
\ll \ell_{r+1} \Big(\frac{\ell_{r+1}}{10e}\Big)^{\ell_{r+1}} \exp\Big(60 \sum_{p\in P_{r+1}} \frac{a(p)^2}{p}\Big)
 \ll \ell_{r+1} \Big(\frac{\ell_{r+1}}{10e}\Big)^{\ell_{r+1}} \exp(\tfrac 35 \ell_{r+1}). 
$$ 

Using the above estimate together with \eqref{6.1} and \eqref{6.3}, we conclude that 
\begin{align*}
\sum_{d \in {\mathcal E}(\kappa,a)} &L(\tfrac 12,E_d)  \prod_{j=1}^{r} {\mathcal A}_j(d) 
\Big(\frac{e^2 {\mathcal P}_{r+1}}{\ell_{r+1}}\Big)^{\ell_{r+1}} \Phi\Big(\frac{\kappa d}{X}\Big) 
\\
&\ll {X} e^{-\ell_{r+1}/2}  \prod_{p\in \cup_{j=1}^{r} P_j} 
\Big( \sum_{j=0}^{\infty} \frac{a(p)^{2j}}{p^j} \frac{(k-1)^{2j}}{(2j)!}g(p^{2j}) + \sum_{j=0}^{\infty} \frac{a(p)^{2j+2}}{p^{j+1}} 
\frac{(k-1)^{2j+1}}{(2j+1)!} g(p^{2j+1})\Big).
\end{align*}
Using Lemma \ref{RS}, we check that for all $0\le r\le R-1$ the above is 
$$ 
\ll X \exp(-\tfrac{\ell_{r+1}}{3}) (\log X)^{\frac{k^2-1}{2}}. 
$$ 

A similar argument shows that 
$$ 
\sum_{d \in {\mathcal E}(\kappa,a)} {L(\tfrac 12,E_d)} \prod_{j=1}^{R} {\mathcal A}_j(d) \Phi\Big(\frac{\kappa d}{X}\Big) 
\ll X(\log X)^{\frac{k^2-1}{2}},
$$ 
and summing all these bounds, we obtain our proposition.

 \section{Proof of Proposition \ref{Prop1}}

\noindent The proof of Proposition \ref{Prop1} follows the general plan of the arguments in \cite{So4} and \cite{SoY} (see also 
\cite{Iw});  
therefore, in some places below we have been brief, and suppressed some details.     
Using Lemma \ref{AFE} in \eqref{S} we obtain 
\begin{equation} 
\label{S2} 
{\mathcal S}(X;u,v) = 2\sum_{\substack{{n=1}\\ {(n,N_0)=1}}}^{\infty} \frac{a(n)}{\sqrt{n}} \sum_{\substack{{d\in {\mathcal E}(\kappa,a)}\\ {v|d}}} \chi_d(nu)W\Big(\frac{n}{\kappa d}\Big) \Phi\Big(\frac{\kappa d}{X}\Big).
\end{equation} 
The inner sum over $d$ in \eqref{S2} runs over multiples of $v$ that are square-free and lying in the progression $a\pmod{N_0}$.  
Thus, using \eqref{S21}, and writing $d=kv\beta \alpha^2$, we see that the sum over $d$ in \eqref{S2} equals 
\begin{equation} 
\label{S3} 
\sum_{\beta|v} \sum_{(\alpha, vN_0)=1} \mu(\beta)\mu(\alpha)\Big(\frac{v\beta \alpha^2}{nu}\Big) \sum_{k \equiv a\overline{v\beta\alpha^2} \pmod{N_0}} 
\Big(\frac{k}{nu}\Big) W\Big(\frac{n}{\kappa k v\beta\alpha^2}\Big) \Phi\Big(\frac{\kappa k v\beta\alpha^2}{X}\Big). 
\end{equation} 

Let $Y>1$ be a parameter to be chosen later.  We distinguish the cases $\beta\alpha^2 >Y$ and $\beta \alpha^2 \le Y$.  
First we bound the contribution of the terms with $\beta \alpha^2 >Y$; the main term will arise from the case $\beta \alpha^2\le Y$.  

\subsection{The terms with $\alpha^2 \beta > Y$}

Consider the contribution of the terms to \eqref{S3} with $\alpha^2 \beta>Y$ and sum that 
over $n$ (as in \eqref{S2}).   Thus the total contribution of such terms to \eqref{S2} is  bounded by 
\begin{equation} 
\label{10.0} 
\sum_{\beta | v} \sum_{\substack{(\alpha, v N_0) = 1 \\
\alpha^2 \beta > Y}}
\sum_{k v \beta \alpha^2 \equiv a \pmod{N_0}} 
\Phi \Big( \frac{\kappa k v \beta \alpha^2}{X} \Big ) \Big|
 \sum_{(n,N_0) = 1} \frac{a(n)}{\sqrt{n}}
\Big ( \frac{k v \beta \alpha^2}{n} \Big ) W \Big ( \frac{n}{
\kappa k v \beta \alpha^2} \Big )\Big|.
\end{equation} 
Now using the definition of $W$, the sum over $n$ above can be rewritten as 
\begin{equation} 
\label{10.1}
\frac{1}{2 \pi i} \int_{(\epsilon)} L_a(s+\tfrac 12) \Gamma(s) \Big(\frac{\sqrt{N}\kappa kv\beta \alpha^2}{2\pi }\Big)^s 
\sum_{(n,N_0)=1} \frac{a(n)}{n^{\frac 12+s}} \Big(\frac{kv\beta \alpha^2}{n} \Big) ds. 
\end{equation} 
Write the discriminant $4 k v \beta \alpha^2$ as  $k_1 k_2^2$
with $k_1$ a fundamental discriminant.  Note that $\alpha \beta$ must divide $k_2$.   Since 
$n$ is odd, $\chi_{kv\beta\alpha^2}(n)=\chi_{k_1k_2^2}(n)$, and so the 
sum over $n$ above may be expressed as $L(\tfrac 12+s,E_{k_1})$ up to Euler factors coming from 
primes dividing $k_2$ and $N_0$ (and these factors are at most $X^{\epsilon}$ in size).   Thus the 
quantity in \eqref{10.1} may be bounded by 
$$ 
\ll X^{\epsilon} \int_{(\epsilon)} |\Gamma(s) L(\tfrac 12+s, E_{k_1})| |ds|. 
$$ 
Using this in \eqref{10.0} we obtain a bound 
$$ 
\ll X^{\epsilon} 
\sum_{\beta|v} \sum_{\alpha^2 \beta>Y} \sum_{\alpha\beta |k_2} \ \  \sumflat_{|k_1|\le X^{1+\epsilon}/k_2^2} 
\int_{(\epsilon)} |\Gamma(s) L(\tfrac 12+s,E_{k_1})| |ds|, 
$$ 
where the $\flat$ indicates a sum over fundamental discriminants.  
  By an application of Heath-Brown's large sieve for quadratic characters 
  (see \cite{HB}, and also Corollary 2.5 of \cite{SoY}) this is 
  \begin{equation} 
  \label{10.2} 
  \ll X^{\epsilon} \sum_{\beta |v} \sum_{\alpha^2 \beta >Y} \sum_{\substack{\alpha \beta |k_2 \\ k_2 \le X^{\frac 12+\epsilon}} } 
  \frac{X}{k_2^2}  \ll X^{1+\epsilon} \sum_{\beta|v} \sum_{\alpha^2 \beta >Y} \frac{1}{\alpha^2 \beta^2}  
  \ll \frac{X^{1+\epsilon}}{\sqrt{Y}}. 
 \end{equation}

\subsection{The terms with $\beta \alpha^2\le Y$: Analysis of the main term} 

Now we turn to the terms in \eqref{S3} with $\beta \alpha^2 \le Y$.  
Put 
$$
F(\xi;x,y) = W\Big( \frac{y}{\kappa \xi}\Big) \Phi\Big(\frac{\kappa \xi}{x}\Big). 
$$ 
Applying Poisson summation (Lemma \ref{Poisson}) to the sum over $k$ in \eqref{S3} we get 
\begin{equation}
\label{S4}
\frac{1}{N_0 nu} \Big(\frac{N_0}{nu}\Big) \sum_{\ell} {\widehat F}\Big(\frac{\ell}{N_0nu};\frac{X}{v\beta\alpha^2},\frac{n}{v\beta\alpha^2}\Big) \tau_{\ell}(nu) e\Big(\frac{\ell a\overline{v\beta\alpha^2 nu}}{N_0}\Big). 
\end{equation} 

The main term arises from $\ell=0$ in \eqref{S4}, which we now analyze.  Note that $\tau_0(nu)=0$ unless $nu$ is a perfect 
square when it equals $\phi(nu)$.   Thus the main term for ${\mathcal S}(X;u,v)$ is 
$$ 
\frac{2}{N_0} 
\sum_{\beta|v} \sum_{\substack{(\alpha,uvN_0)=1\\ \beta\alpha^2 \le Y}} \mu(\beta)\mu(\alpha) \sum_{\substack{{(n,\alpha\beta N_0) =1} \\ {nu=\square}}}\frac{\phi(nu)}{nu} \frac{a(n)}{\sqrt{n}} 
{\widehat F}\Big(0; \frac{X}{v\beta\alpha^2}, \frac{n}{v\beta \alpha^2}\Big).
$$ 
We add back the terms with $\alpha^2 \beta >Y$ above.  Since 
$$
{\hat F}\Big(0;\frac{X}{v\beta \alpha^2}, \frac{n}{v\beta\alpha^2}\Big) =\frac{X}{v\beta \alpha^2}  \int_0^{\infty} \Phi(\xi) 
W\Big(\frac{n}{X\xi}\Big) d\xi
 \ll \frac{X}{v\beta \alpha^2} e^{-n/(X\sqrt{N})}, 
 $$
 the sum over $n$ is at most $X^{1+\epsilon}/(\sqrt{u_1} v\beta \alpha^2)$ and adding this over 
 the terms with $\beta\alpha^2 >Y$ contributes an error of $O(X^{1+\epsilon}/(\sqrt{u_1Y}v))$.   Note that this error is 
 smaller than the error term in \eqref{10.2}.  

Next, using the definition of $W$, we obtain that for any $c>0$ 
\begin{align*}
{\hat F}\Big(0; \frac{X}{v\beta\alpha^2}, \frac{n}{v\beta \alpha^2}\Big) 
& = 
\frac{X}{v\beta\alpha^2} \frac{1}{2\pi i}  \int_{(c)} {\check \Phi}(s) L_a(s+\tfrac 12) \Gamma(s) \Big(\frac{\sqrt{N}X}{2\pi n}\Big)^s ds.
\end{align*} 
Thus the main term in ${\mathcal S}(X;u,v)$ (after extending the sums over $\alpha$ and $\beta$) equals 
\begin{equation} \label{M1}
\frac{2X}{vN_0} \frac{1}{2\pi i} \int_{(c)} {\check \Phi}(s) L_a(s+\tfrac 12) \Gamma(s) \Big(\frac{\sqrt{N}X}{2\pi}\Big)^s 
\sum_{\beta|v} \sum_{(\alpha,uvN_0)=1} \frac{\mu(\beta)\mu(\alpha)}{\beta\alpha^2} \sum_{\substack{{(n,\alpha\beta N_0)=1}\\ {nu=\square}}} \frac{a(n)}{n^{s+\frac 12}} \frac{\phi(nu)}{nu} ds.
\end{equation} 
A little calculation shows that, for a given $n$ coprime to $N_0$, 
$$ 
\sum_{\substack{{\beta|v} \\ {(\beta,n)=1}}} \frac{\mu(\beta)}{\beta} \sum_{(\alpha,uvnN_0)=1} \frac{\mu(\alpha)}{\alpha^2} 
\frac{\phi(nu)}{nu} = \prod_{p\nmid N_0} \Big(1-\frac 1{p^2}\Big) 
\prod_{p| uvn} \Big(1+\frac 1p\Big)^{-1}. 
$$
Thus 
$$ 
\sum_{\beta|v} \sum_{(\alpha,uvN_0)=1}  \frac{\mu(\beta)\mu(\alpha)}{\beta\alpha^2} \sum_{\substack{{(n,\alpha\beta N_0)=1}\\ {nu=\square}}} \frac{a(n)}{n^{s+\frac 12}} \frac{\phi(nu)}{nu}  = 
 \prod_{p\nmid N_0} \Big(1-\frac{1}{p^2}\Big) \sum_{\substack{{(n,N_0)=1}\\ {nu=\square} } } 
 \frac{a(n)}{n^{s+\frac 12} } \prod_{p|uvn} \Big(1+\frac 1p\Big)^{-1}. 
 $$ 
 If $u=u_1u_2^2$ with $u_1$ square free then the condition that $nu$ is a square is the same as writing $n=u_1m^2$.  Thus the RHS above may be expressed as 
\begin{equation} 
\label{G1} 
\frac{a(u_1)}{u_1^{\frac 12+s}} L(1+2s, \text{sym}^2 E) {\mathcal G}(1+2s;u,v)
 \end{equation} 
 where ${\mathcal G}(1+2s;u,v) =\prod_p {\mathcal G}_p(1+2s;u,v)$ is an Euler product defined as follows: 
 If $p|N_0$ the Euler factor ${\mathcal G}_p$ is the inverse of the corresponding Euler factor for $L(1+2s,\text{sym}^2 E)$.  
 If $p|u_1$ then ${\mathcal G}_p(1+2s;u,v) = (1-1/p)(1-1/p^{1+2s})$.  If $p|uv$ but $p\nmid u_1$, then ${\mathcal G}_p(1+2s;u,v) = (1-1/p) (1-1/p^{2+4s})$.  Finally, if $p\nmid uvN_0$, then 
 $$ 
{\mathcal G}_p(1+2s;u,v)= \Big(1-\frac 1p \Big) \Big(1- \frac{1}{p^{1+2s}}\Big) \Big(1+ \frac{1}{p} \Big(1-\frac{\alpha_p^2}{p^{2s+1}}\Big) \Big(1-\frac{\beta_p^2}{p^{2s+1}}\Big)+ \frac{1}{p^{1+2s}} \Big),
$$ 
where we wrote $a(p)=\alpha_p+\beta_p$ with $\alpha_p\beta_p=1$.  It follows that ${\mathcal G}(1+2s;u,v)$ 
  admits an analytic continuation to the region Re$(s) \ge -\frac 14+\epsilon$ and is bounded there by $(N_0uv)^{\epsilon}$.

Using the above remarks in \eqref{M1}, the integrand there is analytic in Re$(s)>-\frac 14$ (except for a simple pole at $s=0$) and therefore, by moving the line of integration to Re$(s)= -\frac 14+ \epsilon$  we obtain that our main term is 
\begin{equation} 
\label{G2}
\frac{2Xa(u_1)}{vu_1^{\frac 12}N_0} {\check \Phi}(0) L_a(\tfrac 12) L(1,\text{sym}^2 E)  
{\mathcal G}(1;u,v) + O(X^{\frac 34+\epsilon})+O\Big(\frac{X^{1+\epsilon}}{\sqrt{Y}}\Big). 
\end{equation}
This is the main term in our Theorem, and the decomposition of ${\mathcal G}(1;u,v)$ as a 
constant times $g(u) h(v)$ for appropriate multiplicative functions $g$ and $h$ follows from 
our remarks on the Euler factors of ${\mathcal G}$.  

\subsection{The terms with $\beta \alpha^2 \le Y$: Estimating the remainder terms}

Recall that the Fourier transform ${\hat F}(\lambda; x,y)$ (where $\lambda$, $x$ and $y$ are real numbers, with $x$ and $y$ positive) is given by 
$$ 
{\widehat F}(\lambda;x,y) = x \int_0^{\infty} \Phi(\xi) W\Big(\frac{y}{x\xi} \Big) e(-\kappa \lambda x\xi) d\xi.
$$  
Since $W(t)$ and its derivatives decrease  rapidly as $t \to \infty$, we obtain that $|{\hat F}(\lambda;x,y)| 
\ll_A x (x/y)^A$ for any integer $A\ge 0$.  Further, integrating by parts many times, we also find that 
$|{\hat F}(\lambda;x,y)| \ll_A x (|\lambda|y)^{-A}$.   Thus we have 
\begin{equation} 
\label{10.3.1} 
|\widehat{F}(\lambda; x,y)| \ll_A x\min \Big( \Big(\frac{x}{y}\Big)^A, \frac{1}{(|\lambda| y)^A} \Big). 
\end{equation} 
 
 For $\alpha^2 \beta \le Y$ we must bound the contribution of the terms with $\ell\neq 0$ in \eqref{S4} 
 to the quantity in \eqref{S2}.  This is 
 \begin{equation} 
 \label{10.3.2}
 \ll \frac{1}{N_0u}  \sum_{\beta|v} \sum_{\substack{(\alpha,vN_0)=1\\ \alpha^2 \beta\le Y}} \Big|\sum_{\ell \neq 0} 
 \sum_{(n,N_0)=1} \frac{a(n)}{n^{\frac 32}} \Big(\frac{v\beta\alpha^2}{n}\Big) {\widehat F}\Big(\frac{\ell}{N_0nu};\frac{X}{v\beta \alpha^2}, 
 \frac{n}{v\beta \alpha^2} \Big) \tau_{\ell}(nu) e\Big(\frac{\ell a\overline{v\beta\alpha^2 nu}}{N_0}\Big)\Big|.
 \end{equation} 
 
 First we show that the terms $|\ell| > N_0 uvY X^{\epsilon}$ make a negligible contribution above.  Using \eqref{10.3.1} 
 we see that 
 $$ 
 \Big| {\widehat F}\Big(\frac{\ell}{N_0nu};\frac{X}{v\beta\alpha^2}, \frac{n}{v\beta\alpha^2}\Big) \Big| 
 \ll_A \frac{X}{v\beta \alpha^2} \Big(\frac{Nuv\beta \alpha^2}{|\ell|}\Big)^A \Big(\frac{X}{n}\Big)^2 \ll 
 \frac{1}{X \ell^2 n^2}, 
 $$ 
 by choosing $A$ appropriately large.  Using this in \eqref{10.3.2}, we deduce that the contribution of the terms 
 with $|\ell| > N_0 uvY X^{\epsilon}$ is $\ll X^{-1}$, which is indeed negligible.

 Now suppose $1\le |\ell |\le N_0 uvY X^{\epsilon}$, and consider the sum over $n$ in \eqref{10.3.2}.   We 
 remove the $e(\ell a\overline{v\beta\alpha^2nu}/N_0)$ term by introducing Dirichlet characters $\psi \pmod {N_0}$: 
 $$ 
 e\Big(\frac{\ell a\overline{v\beta \alpha^2nu}}{N_0} \Big) = 
 \sum_{\psi\pmod{N_0}} \psi(n) \Big(\frac{1}{\phi(N_0)} \sum_{b\pmod {N_0}} \overline{\psi(b)} e\Big(\frac{\ell a\overline{v\beta\alpha^2bu}}{N_0}\Big) \Big).   
 $$ 
 Since the sum over $b$ above is trivially bounded by $\phi(N_0)$, we are reduced to the problem of estimating 
 \begin{equation} 
 \label{10.3.7} 
 \sum_{\psi \pmod {N_0}} \Big| \sum_{(n,N_0)=1} \frac{a(n)}{n^{\frac 32}} \Big(\frac{v\beta\alpha^2}{n}\Big) 
 \psi(n) \tau_{\ell}(nu) {\widehat F}\Big( \frac{\ell }{N_0nu};\frac{X}{v\beta \alpha^2}, \frac{n}{v\beta \alpha^2}\Big)\Big|. 
 \end{equation} 
 
 Now we pass to Mellin transforms in order to handle the sum over $n$ above.   For a complex number $s$ with Re$(s)>0$, put  
 \begin{equation} 
 \label{10.3.8}
 {\widetilde F}(s;\ell,\beta\alpha^2) = 
 \int_0^{\infty} {\widehat F}\Big(\frac{\ell}{N_0tu}; \frac{X}{v\beta \alpha^2}, \frac{t}{v\beta\alpha^2}\Big) t^{s-1} dt, 
 \end{equation} 
 which, using the definition of ${\widehat F}$, may be expressed as 
 \begin{equation} 
 \label{10.3.9} 
\frac{X^{s+1}}{v\beta\alpha^2}  {\check \Phi}(s) \int_0^{\infty} W\Big(\frac 1{y} \Big) e\Big(-\frac{\kappa \ell y}{N_0uv\beta\alpha^2}\Big) \frac{dy}{y^{s+1}}.
 \end{equation}  
 By Mellin inversion, the sum over $n$ in \eqref{10.3.7} is, for any $c>0$,  
 \begin{equation} 
 \label{10.3.10} 
 \frac{1}{2\pi i} \int_{(c)} {\widetilde F}(s;\ell,\beta\alpha^2)  \sum_{(n,N_0)=1} \frac{a(n)}{n^{\frac 32+s}} 
 \Big(\frac{v\beta\alpha^2}{n}\Big) \psi(n) \tau_{\ell}(nu) ds. 
 \end{equation}  
 Now from \eqref{10.3.9}, and using that $W(\xi)=L_a(\frac 12)+O(\xi^{\frac 12-\epsilon})$ as $\xi\to 0$,  we may obtain an analytic continuation of ${\widetilde F}(s;\ell, \beta\alpha^2)$ to 
 the region Re$(s)>-\frac 12+\epsilon$, and that it is bounded in that region by $\ll_A X^{1+\text{Re}(s)}/(v\beta \alpha^2 (1+|s|)^A)$ for any $A>0$.  Furthermore, expressing $\tau_\ell$ in terms of the multiplicative $G_{\ell}$ and $G_{-\ell}$ and 
 using Lemma 5, we see that the sum over $n$ above may be expressed in terms of $G_{\pm \ell}(u)$ times $L(1+s,E \times \psi \chi_{\pm v\beta \ell})$ times certain Euler factors at primes $p| N_0 \ell v\beta\alpha^2$.  
 Using the convexity bound for $L$-functions, we may bound this quantity in the region Re$(s)>-\frac 12+\epsilon$ by $\ll uX^{\epsilon} (v\beta \ell(1+|s|))^{\frac 12+\epsilon}$.   Therefore we conclude that the quantity in \eqref{10.3.10} is bounded 
 by 
 $$ 
 \ll \frac{X^{\frac 12+\epsilon}}{v^{\frac 12} \beta^{\frac 12} \alpha^2} u \ell^{\frac 12+\epsilon}. 
 $$  
 Using this estimate in \eqref{10.3.7} and summing over all $1\le |\ell| \le N_0uvY X^{\epsilon}$, we conclude that 
 these terms contribute to \eqref{10.3.2} an amount bounded by 
 $ \ll   u^{\frac 32} v Y^{\frac 32} X^{\frac 12+\epsilon}$.  We conclude that the contribution of  the remainder terms 
 arising from $\beta \alpha^2 \le Y$ is 
 \begin{equation}
 \label{10.3.11} 
 \ll  u^{\frac 32} v Y^{\frac 32} X^{\frac 12+\epsilon} + X^{-1} \ll  u^{\frac 32} v Y^{\frac 32} X^{\frac 12+\epsilon}. 
 \end{equation}

 %

\subsection{Completion of the proof}

Choose $Y=X^{\frac 14}v^{-\frac 12} u^{-\frac 34}$, and the proposition follows in view of \eqref{10.2}, \eqref{G2} and \eqref{10.3.11}. 

\bibliography{MomRefs}{}

\begin{thebibliography}{10}

\bibitem{ChLi2}
Vorrapan Chandee and Xiannan Li.
\newblock The eighth moment of {D}irichlet {$L$}-functions.
\newblock {\em preprint, arXiv:1303.4482}, 2013.

\bibitem{Coates}
John Coates, Yongxiong Li, Ye~Tian, and Shuai Zhai.
\newblock Quadratic twists of elliptic curves.
\newblock {\em preprint, arXiv:1312.3884}, 2013.

\bibitem{CFKRS}
J.~B. Conrey, D.~W. Farmer, J.~P. Keating, M.~O. Rubinstein, and N.~C. Snaith.
\newblock Integral moments of {$L$}-functions.
\newblock {\em Proc. London Math. Soc. (3)}, 91(1):33--104, 2005.

\bibitem{CIS}
J.~B. Conrey, H.~Iwaniec, and K.~Soundararajan.
\newblock The sixth power moment of {D}irichlet {$L$}-functions.
\newblock {\em Geom. Funct. Anal.}, 22(5):1257--1288, 2012.

\bibitem{CKRS}
J.~B. Conrey, J.~P. Keating, M.~O. Rubinstein, and N.~C. Snaith.
\newblock Random matrix theory and the {F}ourier coefficients of
  half-integral-weight forms.
\newblock {\em Experiment. Math.}, 15(1):67--82, 2006.

\bibitem{Del1}
Christophe Delaunay.
\newblock Moments of the orders of {T}ate-{S}hafarevich groups.
\newblock {\em Int. J. Number Theory}, 1(2):243--264, 2005.

\bibitem{Del2}
Christophe Delaunay.
\newblock Heuristics on class groups and on {T}ate-{S}hafarevich groups: the
  magic of the {C}ohen-{L}enstra heuristics.
\newblock In {\em Ranks of elliptic curves and random matrix theory}, volume
  341 of {\em London Math. Soc. Lecture Note Ser.}, pages 323--340. Cambridge
  Univ. Press, Cambridge, 2007.

\bibitem{DelWat}
Christophe Delaunay and Mark Watkins.
\newblock The powers of logarithm for quadratic twists.
\newblock In {\em Ranks of elliptic curves and random matrix theory}, volume
  341 of {\em London Math. Soc. Lecture Note Ser.}, pages 189--193. Cambridge
  Univ. Press, Cambridge, 2007.

\bibitem{DGH}
Adrian Diaconu, Dorian Goldfeld, and Jeffrey Hoffstein.
\newblock Multiple {D}irichlet series and moments of zeta and {$L$}-functions.
\newblock {\em Compositio Math.}, 139(3):297--360, 2003.

\bibitem{EFKAMM}
Bruno Eckhardt, Shmuel Fishman, Jonathan Keating, Oded Agam, J{\" o}rg Main,
  and Kirsten M{\" u}ller.
\newblock Approach to ergodicity in quantum wave functions.
\newblock {\em Phys. Rev. E}, 52:5893--5903, 1995.

\bibitem{Gold1}
Dorian Goldfeld.
\newblock Conjectures on elliptic curves over quadratic fields.
\newblock In {\em Number theory, {C}arbondale 1979 ({P}roc. {S}outhern
  {I}llinois {C}onf., {S}outhern {I}llinois {U}niv., {C}arbondale, {I}ll.,
  1979)}, volume 751 of {\em Lecture Notes in Math.}, pages 108--118. Springer,
  Berlin, 1979.

\bibitem{GrSo}
Andrew Granville and K.~Soundararajan.
\newblock Sieving and the {E}rd{\H o}s-{K}ac theorem.
\newblock In {\em Equidistribution in number theory, an introduction}, volume
  237 of {\em NATO Sci. Ser. II Math. Phys. Chem.}, pages 15--27. Springer,
  Dordrecht, 2007.

\bibitem{Har}
A.~J. Harper.
\newblock Sharp conditional bounds for moments of the {R}iemann zeta function.
\newblock {\em Preprint, arXiv:1305.4618}, 2013.

\bibitem{HB3}
D.~R. Heath-Brown.
\newblock Fractional moments of the {R}iemann zeta function.
\newblock {\em J. London Math. Soc. (2)}, 24(1):65--78, 1981.

\bibitem{HB}
D.~R. Heath-Brown.
\newblock A mean value estimate for real character sums.
\newblock {\em Acta Arith.}, 72(3):235--275, 1995.

\bibitem{Ho}
R.~D. Hough.
\newblock The distribution of the logarithm in an orthogonal and a symplectic
  family of $l$-functions.
\newblock {\em Forum Mathematicum}, 26:523--546, 2014.

\bibitem{HY}
C.~P. Hughes and Matthew~P. Young.
\newblock The twisted fourth moment of the {R}iemann zeta function.
\newblock {\em J. Reine Angew. Math.}, 641:203--236, 2010.

\bibitem{Iw}
Henryk Iwaniec.
\newblock On the order of vanishing of modular {$L$}-functions at the critical
  point.
\newblock {\em S\'em. Th\'eor. Nombres Bordeaux (2)}, 2(2):365--376, 1990.

\bibitem{IK}
Henryk Iwaniec and Emmanuel Kowalski.
\newblock {\em Analytic number theory}, volume~53 of {\em American Mathematical
  Society Colloquium Publications}.
\newblock American Mathematical Society, Providence, RI, 2004.

\bibitem{KaSa}
Nicholas~M. Katz and Peter Sarnak.
\newblock {\em Random matrices, {F}robenius eigenvalues, and monodromy},
  volume~45 of {\em American Mathematical Society Colloquium Publications}.
\newblock American Mathematical Society, Providence, RI, 1999.

\bibitem{KeSn1}
J.~P. Keating and N.~C. Snaith.
\newblock Random matrix theory and {$L$}-functions at {$s=1/2$}.
\newblock {\em Comm. Math. Phys.}, 214(1):91--110, 2000.

\bibitem{KeSn2}
J.~P. Keating and N.~C. Snaith.
\newblock Random matrix theory and {$\zeta(1/2+it)$}.
\newblock {\em Comm. Math. Phys.}, 214(1):57--89, 2000.

\bibitem{LS}
Wenzhi Luo and Peter Sarnak.
\newblock Quantum variance for {H}ecke eigenforms.
\newblock {\em Ann. Sci. \'Ecole Norm. Sup. (4)}, 37(5):769--799, 2004.

\bibitem{Miller}
Robert~L. Miller.
\newblock Proving the {B}irch and {S}winnerton-{D}yer conjecture for specific
  elliptic curves of analytic rank zero and one.
\newblock {\em LMS J. Comput. Math.}, 14:327--350, 2011.

\bibitem{RaSo}
Maksym Radziwi{\l}l and Kannan Soundararajan.
\newblock Continuous lower bounds for moments of zeta and {$L$}-functions.
\newblock {\em Mathematika}, 59(1):119--128, 2013.

\bibitem{Rubin}
Karl Rubin.
\newblock Fudge factors in the {B}irch and {S}winnerton-{D}yer conjecture.
\newblock In {\em Ranks of elliptic curves and random matrix theory}, volume
  341 of {\em London Math. Soc. Lecture Note Ser.}, pages 233--236. Cambridge
  Univ. Press, Cambridge, 2007.

\bibitem{RuSo1}
Z.~Rudnick and K.~Soundararajan.
\newblock Lower bounds for moments of {$L$}-functions.
\newblock {\em Proc. Natl. Acad. Sci. USA}, 102(19):6837--6838, 2005.

\bibitem{RuSo2}
Z.~Rudnick and K.~Soundararajan.
\newblock Lower bounds for moments of {$L$}-functions: symplectic and
  orthogonal examples.
\newblock In {\em Multiple {D}irichlet series, automorphic forms, and analytic
  number theory}, volume~75 of {\em Proc. Sympos. Pure Math.}, pages 293--303.
  Amer. Math. Soc., Providence, RI, 2006.

\bibitem{So4}
K.~Soundararajan.
\newblock Nonvanishing of quadratic {D}irichlet {$L$}-functions at
  {$s=\frac12$}.
\newblock {\em Ann. of Math. (2)}, 152(2):447--488, 2000.

\bibitem{SoY}
K.~Soundararajan and Matthew~P. Young.
\newblock The second moment of quadratic twists of modular {$L$}-functions.
\newblock {\em J. Eur. Math. Soc. (JEMS)}, 12(5):1097--1116, 2010.

\bibitem{So}
Kannan Soundararajan.
\newblock Moments of the {R}iemann zeta function.
\newblock {\em Ann. of Math. (2)}, 170(2):981--993, 2009.

\bibitem{Zhao}
Peng Zhao.
\newblock Quantum variance of {M}aass-{H}ecke cusp forms.
\newblock {\em Comm. Math. Phys.}, 297(2):475--514, 2010.

\end{thebibliography}
\bibliographystyle{plain}

\end{document}